\documentclass[leqno,11pt]{amsart}
\usepackage[T1]{fontenc}
\usepackage{graphicx}
\usepackage{amscd}
\usepackage{amsthm}
\usepackage{amsmath}
\usepackage{caption}
\usepackage{amsfonts}
\usepackage{amssymb}
\usepackage{mathrsfs}
\usepackage{multicol}
\usepackage{color}
\usepackage{hyperref}
\hypersetup{
    colorlinks=true,
    linkcolor=magenta,
    citecolor=blue,
}
\usepackage{enumitem}

\usepackage{todonotes}

\usepackage{mathtools} 
\mathtoolsset{showonlyrefs}%Permette di usare sempre equation/align, ma numera solo le formule che vengono richiamate almeno una volta.
\usepackage[hyperpageref]{backref}
\usepackage[left=2cm,right=2cm,top=2.72cm,bottom=2.72cm]{geometry}
\usepackage{booktabs}
\usepackage{appendix}

\newcommand{\NN}{\mathbb N}

\newcommand{\R}{\mathbb R}
\newcommand{\CC}{\mathbb C}
\newcommand{\C}{\mathbb C}
\newcommand{\be}{\begin{equation}}
\newcommand{\ee}{\end{equation}}
\newcommand{\ben}{\begin{eqnarray*}}
\newcommand{\een}{\end{eqnarray*}}

\newcommand{\ve}{\varepsilon}

\newcommand{\bs}{\boldsymbol}

\newcommand{\cB}{{\mathcal B}}
\newcommand{\cE}{{\mathcal E}}
\newcommand{\cH}{{\mathcal H}}
\newcommand{\cK}{{\mathcal K}}
\newcommand{\cD}{{\mathcal D}}
\newcommand{\cI}{{\mathcal I}}

\newcommand{\eps}{\varepsilon}
\newcommand{\supp}{\mathrm{supp}\,}
\newcommand{\loc}{\mathrm{loc}\,}

\numberwithin{equation}{section}
\newtheorem{theorem}{Theorem}[section]
\newtheorem{lemma}[theorem]{Lemma}
\newtheorem{remark}[theorem]{Remark}

\newtheorem{proposition}[theorem]{Proposition}

\allowdisplaybreaks
\definecolor{darkgreen}{rgb}{0.09, 0.45, 0.27}
\definecolor{debianred}{rgb}{0.84, 0.04, 0.33}

\def\e{{\rm e}}
 \def\dd{\, {\rm d}}

\newcommand{\weakto}{\rightharpoonup}

\newcommand{\expAplus}{\e^{i(x-y)\cdot A\left(\frac{x+y}{2}\right)}}
\newcommand{\expAminus}{\e^{-i(x-y)\cdot A\left(\frac{x+y}{2}\right)}}

\newcommand{\divv}{\text{div}}

\allowdisplaybreaks

\begin{document}

\title[Fractional magnetic $p$-Laplacian and applications]{Qualitative properties of the fractional magnetic $p$-Laplacian and applications to critical quasilinear problems}

\author[L. Baldelli]{Laura Baldelli}
    \address[L. Baldelli]{\newline\indent
	Institute for Analysis
	\newline\indent
	Karlsruhe Institute of Technology (KIT)
	\newline\indent
	Englerstraße 2
	\newline\indent
	D-76128 Karlsruhe, Germany}
    \email{\href{mailto:laura.baldelli@kit.edu}{laura.baldelli@kit.edu}
}

\author[F. Bernini]{Federico Bernini}
    \address[F. Bernini]{\newline\indent
	Dipartimento di Matematica e Applicazioni
	\newline\indent
	Università degli Studi di Milano-Bicocca
	\newline\indent
	Via R. Cozzi, 55
	\newline\indent
	I-20125 Milan, Italy.}
    \email{\href{mailto:federico.bernini@unimib.it}{federico.bernini@unimib.it}
}



\begin{abstract}
We investigate the fractional magnetic $p$-Laplacian operator in the physical dimension case $N=3$, with $0<s<1<p$ and $sp<3$. Our goal is twofold. First, we define and study suitable functional settings for such operator proving significant properties. Then we get the existence of weak solutions for some quasilinear equations involving a weighted critical and subcritical power type nonlinearity. Our technique relies on variational methods and faces various difficulties: the complex quasilinear framework due to the presence of an external magnetic potential, the nonlocal setting, which entails appropriate tools, and the lack of compactness, which requires concentration compactness arguments. In this direction, we state a new concentration compactness principle in the quasilinear magnetic setting that seems to be missing in the literature.
\end{abstract}

\maketitle

\section{Introduction}
The study of Schr\"odinger-type equations attracted the interest of nonlinear analysis in recent decades, giving rise to an impressive body of literature on the subject. The classical Schr\"odinger equation is governed by the Laplacian operator together with the electric (or Coulomb) potential of the system under study, i.e. $-\Delta + V$. From a mathematical point of view, the presence of the function $V$ is seen as a zero-order perturbation of the Laplacian, and its behavior strongly impacts the techniques and theories required to approach the study of such problems. The variational method is a cornerstone approach for identifying (nontrivial) critical points of the functional associated; these points are subsequently identified as weak solutions to the original problem. This framework can be further developed using minimax methods, such as the Mountain Pass Theorem, or topological tools like the Krasnosel’skii genus to handle problems involving existence and multiplicity. 
Although electric potentials $V$ are typically handled within real weighted Sobolev spaces, the inclusion of magnetic effects in Schr\"odinger equations introduces significant mathematical challenges. From a technical standpoint, the magnetic field acts as a first-order differential perturbation of the Laplacian, appearing as an \textit{imaginary} component that necessitates the use of complex-valued functional spaces. This shift leads to the loss of classical tools, such as the maximum principle. Consequently, due to the inherent complexity of this operator, developing its setting has been neither immediate nor intuitive.

Despite these obstacles, the study of the \textit{magnetic Laplacian}, formally defined below, has achieved significant success in recent years. Take $u:\R^3 \to \C$ any smooth function and define
\begin{equation}\label{magn:lapl}
    -\Delta_A u :=-\divv_A(\nabla_A u)= -(\nabla -i A)^2u=-\Delta u + iu\divv A + 2iA\cdot \nabla u + |A|^2 u,
\end{equation}
where, differently from the Laplacian, the classical spatial derivative and divergence are replaced by the covariant derivative and divergence, respectively
\begin{equation}
    \nabla_A = \nabla - iA, \qquad \divv_A=\divv+iA\,.
\end{equation} 
Phisically, the magnetic Laplacian describes the magnetic kinetic energy, up to constants, of a quantum particle subject to the magnetic potential $A$, while the total energy of the quantum particle is described by the Hamiltonian $-\Delta_A+V$ where, as above, $V$ is the electric (or Coulomb) potential of the system. Actually, the magnetic potential $A$ is related to its magnetic vector field $B$, which satisfies the Maxwell equation $\nabla \times B=0$, by the relation $B = \nabla \times A$ in dimension three. Note that
given $B$, the potential $A$ is not uniquely determined: this is called gauge invariance in physics. Motivated by the physical argument above, in this paper we will consider just the $3-$dimensional case, even if, mathematically speaking, everything we are going to present works for every dimension $N \geq 2$, clearly by suitable modification, in particular on how to consider the magnetic potential. Indeed, if $N=3$, then $B=\nabla \times A$, while, for any other $N$, we should think $B$ as the $2-$form given by $B_{ij}=\partial_jA_k - \partial_kA_k$.\\
From the firsts papers \cite{ReedSimon1975, EstebanLions, AvronHerbstSimon} dealing with the magnetic Laplacian operator \eqref{magn:lapl}, different approaches have been taken to cover different PDE driven by \eqref{magn:lapl} and its generalization, we refer to \cite{pankov, ArioliSzulkin2003, chabszulkin, kurata, cingolani, Ambrosio_schrodinger, BoNyVS2019, CiClSe2012}.
The magnetic intricacies are probably behind the fact that the nonlinear generalizations based on the $p$-Laplacian, i.e.
\begin{equation}
-\Delta_{A,p}u:=\divv_A(|\nabla_Au|^{p-2}\nabla_Au), \qquad u:\R^3 \to \C,
\end{equation}
were not considered in the literature until very recently. Indeed, characterizations of magnetic Sobolev spaces were made in \cite{NguyenPinamontiSquassinaVecchi_2018}, Hardy-type inequality has been proved in \cite{ckl} (see also \cite{ChenTang}) and variational arguments were used in \cite{bfk} (see also \cite{WeiSong}).

\medskip

Moving to the fractional world, the magnetic generalization of \eqref{magn:lapl} was introduced by d'Avenia and Squassina in \cite{dAveniaSquassina}, and in \cite{BerninidAvenia} for  the pseudorelativistic setting, starting from the work of Ichinose and Tamura \cite{IchinoseTamura} where $s=1/2$. Namely, the authors in \cite{dAveniaSquassina}, following the construction of \cite{IchinoseTamura}, considered for every $s \in (0,1)$, the following operator
\begin{equation}
    \label{dAveniaSquassina:def}
    \cH^s_Au(x)=\lim_{\eps \searrow 0} \int_{B_\eps^c(x)}(u(x)-\e^{i(x-y)\cdot A\left(\frac{x+y}{2}\right)})\mu_s^m(y-x) \dd y + m^{2s}u(x),\qquad m\ge 0,
\end{equation}
where $\mu^m_s$ is a L\'evy measure defined as 
(cf. \cite[Eq. (1.9)]{BerninidAvenia})
\begin{equation}
    \label{Levy:measure:p=2}
    \mu_m^s(y) \dd y=
    \begin{cases}
        \frac{c_s}{|y|^{3+2s}}\dd y, &m=0,\\
        C_sm^{\frac{3+2s}{2}}\frac{\cK_{\frac{3+2s}{2}}(m|y|)}{|y|^{\frac{3+2s}{2}}}\dd y, &m>0,
    \end{cases}
\end{equation}
and $\cK_\nu(\zeta)$ denotes the modified Bessel function of the second kind (see \cite[Appendix A]{BerninidAvenia} for a brief and self-contained recall of the definition and the main properties needed in this environment) with 
\begin{equation}
    C_s:= c_s\frac{2^{-\frac{3+2s}{2}+1}}{\Gamma\left(\frac{3+2s}{2}\right)} = \frac{s2^{\frac{2s-1}{2}}}{\pi^\frac32\Gamma(1-s)}.
\end{equation}

Then, in the same paper d’Avenia and Squassina dealt with the case $m=0$ and $s \in (0,1)$, providing the suitable functional setting together with some properties, such as embeddings, vanishing, and dichotomy lemmas, useful in order to carry on variational arguments for studying PDEs driven by the \textit{fractional magnetic laplacian}
\begin{equation}
    \label{dAveniaSquassina:def:2}
    (-\Delta)^s_Au(x):=c_s\lim_{\eps \searrow 0} \int_{B_\eps^c(x)}\frac{u(x)-\e^{i(x-y)\cdot A\left(\frac{x+y}{2}\right)}u(y)}{|x-y|^{3+2s}} \dd y.
\end{equation}
The case $m>0$ and $s \in (0,1)$ has been recently treated in \cite{BerninidAvenia}, where similar results have been proved. Furthermore, limits results for the fractional parameter $s$ going to $1^-$ have been shown, proving both a Bourgain-Brezis-Mironescu-type result and the asymptotic behavior of the operator. In both papers, applications to semilinear and Choquard equations have been done, by means of the tools proved.

PDEs driven by \eqref{dAveniaSquassina:def:2} attracted attention of many mathematicians, without the intention to be complete, we mention \cite{AffiliValdinoci,Ambrosio2019-1,dAveniaSquassina, FiPiVe} if $m=0$, while for the so-called pseudorelativistic case see \cite{Ambrosio2016,BerninidAvenia,CiSe2018,FaFe,Secchi2019}.

\medskip

In this paper, we aim to generalize, in a way we can say natural, both the quasiliear local case and the fractional case of \cite{dAveniaSquassina} by considering the \textit{fractional magnetic $p$-Laplace operator}: that is, 
for every $x \in \R^3$ we define
\begin{equation}
\label{BaldelliBernini:def}
    (-\Delta)^s_{p,A}u(x):=C_{s,p} \lim_{\eps \searrow 0} \int_{B_\eps^c(x)} \frac{|u(x) - \expAplus u(y)|^{p-2}(u(x)-\expAplus u(y))}{|x-y|^{3+sp}} \dd y,
\end{equation}
where
\begin{equation}\label{csp}
    C_{s,p}:= c_{s,p}\frac{2^{-\frac{3+ps}{2}+1}}{\Gamma\left(\frac{3+ps}{2}\right)}.
\end{equation}
For completeness, we mention \cite{PinamontiSquassinaVecchi:s:to:0}, \cite{PinamontiSquassinaVecchia:s:to:1} where the authors define \eqref{BaldelliBernini:def} and provide the magnetic counterpart of the convergence result of Maz’ya-Shaposhnikova for vanishing fractional orders $s$, namely for $s\searrow 0^+$, and \textit{\`a-la} Bourgain-Brezis-Mironescu, that is for $s$ approaching $1^-$, respectively, for \eqref{BaldelliBernini:def}.
Since we will work with a fixed $s$ and will not address this type of result, we normalize the constant to $1$ throughout the paper.

Although the passage from $p=2$ to a general $p$ does not seem so deep from a formal point of view, it actually creates some problem of what it should be the underlying functional space. Indeed, as clearly shown in \cite{dAveniaSquassina}, if $p=2$ the functional space of definition is (as one clearly expects) a Hilbert space, endowed then with an inner product that gives rise to a norm (cf. \cite[Proposition 2.1]{dAveniaSquassina}). However, for a general $p>1$ this is not longer true, as we already know from the classical theory of Sobolev spaces: indeed, in such a case, the underlying space is nothing more than a Banach space, and some suitable modifications are in order.
\medskip

The aim of the present paper is twofold. In the first part, we provide the entire functional setting for \eqref{BaldelliBernini:def}. 
Clearly, we take inspiration from \cite{dAveniaSquassina}, but we also had to adapt arguments from \cite{BrascoGV}, which deals with the non-magnetic quasilinear fractional spaces, to the magnetic setting. So, combining and adapting the techniques of these two papers, we are able to provide a suitable functional space for \eqref{BaldelliBernini:def}: this is the task of Section \ref{Section:funct:sett}. Due to the several tools we needed to obtain the functional environment, we decided to divide the presentation into three steps. Firstly, we recall some known facts for the nonmagnetic case.

Then, we move to the non-homogeneous magnetic case, where we are going to define the fractional magnetic Sobolev space $W^{s,p}_A(\R^3,\CC)$ and to prove some embedding properties (cf. Lemma \ref{local:cont:emb}, Lemma \ref{Sobolev:emedding:lemma}), and the validity of the gauge invariance.

Lastly, with the help of the previous case, we move on to the homogeneous magnetic Sobolev space $D^{s,p}_A(\R^3,\CC)$, which needs particular attention, as in the nonmagnetic case \cite{BrascoGV}. Indeed, being $sp<3$ in our case, using the fractional Sobolev inequality, it is possible to give a concrete characterization of the completion $D^{s,p}_A(\R^3,\CC)$ as a functional space by a linear isometric isomorphism (see Proposition \ref{isometric:isomorphism}). Then, the embeddings easily follow from the non-homogeneous case. A last important tool of this part is that we are able to show that both the Sobolev constant of the nonmagnetic setting and the magnetic one are indeed equivalent. This is a crucial information that we have heavily used in order to prove a concentration-compactness Theorem in the magnetic setting (see Theorem \ref{Our:concentration:compactness} whose proof is given in Appendix \ref{Appendix}) which, as far as we know, is new in the literature.

\medskip

The second part of the paper concerns the study of PDEs driven by \eqref{BaldelliBernini:def} in the homogeneous magnetic Sobolev space $D^{s,p}_A(\R^3,\CC)$.
Throughout the paper, we consider magnetic potentials $A$ with locally bounded gradient, which is a stronger condition with respect to the one in \cite{bfk} where the local case was considered. This condition seems to naturally arise from the definition of the functional space, see \cite[Proposition 2.2]{dAveniaSquassina}.
In particular, we aim to find existence result of the generalization of the Brezis Nirenberg type equation for the fractional magnetic $p$-Laplacian problem in the entire Euclidean space $\R^3$, namely
\begin{equation}
    \label{main:equation}
    (-\Delta)^s_{p,A}u =\lambda H(x)|u|^{q-2}u+K(x)|u|^{p^*_{s}-2}u\quad \text{ in } \R^3,
\end{equation}
where $0<s<1<p$, $ps<3$, $1<q<p^*_s$, $p^*_s=3p/(3-sp)$ is the \textit{p-fractional Sobolev critical exponent} in dimension three, $\lambda>0$ is a positive parameter, $A:\R^3 \to \R^3$ is a magnetic potential with locally bounded
gradient and $H,K$ are nonnegative nontrivial weights satisfying
\begin{equation}\label{H:hyp}\tag{H}
0\le H\in L^{r}(\R^3, \R), \qquad r=p^*_{s}/(p^*_{s}-q),
\end{equation}
\begin{equation}\label{K:hyp}\tag{K}
0\le K\in L^{\infty}(\R^3, \R)\cap C(\R^3, \R).
\end{equation}

We would like to perform a variational approach, i.e. searching for critical points of the (formally) energy functional $\cE:D^{s,p}_A(\R^3,\CC) \to \R$, which will be introduced in Section \ref{Sect:BN}, associated to \eqref{main:equation}.

One of the reasons why we decided to take into account a Brezis Nirenberg nonlinearity in \eqref{main:equation}, i.e. a critical term perturbed by a subcritical one, concerns its physical relevance (the most notorious example is Yamabe’s problem) and mathematical motivation. Indeed, it permits us to restore compactness that is otherwise lost at the critical level by concentration around points and at infinity due to the lack of 
compactness of the embedding of the Sobolev space into the critical Lebesgue space.

For this reason, the pioneering paper of Brezis and Nirenberg \cite{bn} paved the way in the study of critical nonlinear Schr\"odinger equation of the type \eqref{main:equation} with $A=0$, $p=2$ and $s=1$ in bounded domains. Later, in the entire Euclidean space, existence and multiplicity results for semilinear magnetic ($A\neq 0$, $p=2$), critical problems driven by the local ($s=1$) magnetic Laplacian \eqref{magn:lapl} were gained in \cite{ArioliSzulkin2003, chabszulkin, Ambrosio_schrodinger}. On the other hand, the nonlocal ($0<s<1$) semilinear ($p=2$) counterpart can be find in \cite{MosconiPereraSquassinaYang} for $A=0$, while we mention \cite{BSX, Ambrosio_kirchhoff, Ambrosio2020-3, YangAn} for the magnetic case $A\neq0$. The paper already mentioned \cite{dAveniaSquassina} treated the fractional critical magnetic semilinear case. Finally, passing to the local quasilinear critical magnetic case, i.e. $s=1$, $p\neq 2$ and $A\neq 0$, we are aware about only one paper \cite{bfk} where, in the nonhomogeneous Sobolev space, the existence of at least one nontrivial solution was established. Clearly, The question regarding the multiplicity of solutions is entirely open. 

Equation \eqref{main:equation} represents the ``natural'' fractional generalization of the problem studied in \cite{bfk}. However, since our  setting involves the homogeneous Sobolev space, the techniques developed  in \cite{bfk} cannot be applied straightforwardly. As noted previously, the fractional magnetic $p$-Laplacian \eqref{BaldelliBernini:def} was first introduced in \cite{npsv, NguyenPinamontiSquassinaVecchi_2018} within the context of convergence results. A significant contribution of the present work lies in establishing a rigorous functional framework and proving several fundamental properties of this operator. Furthermore, we introduce the use of variational methods to treat PDEs driven by \eqref{BaldelliBernini:def}, which, to the best of our knowledge, constitutes a novelty.

We should also remark that while an earlier attempt to study  Kirchhoff-type equations involving \eqref{main:equation} appeared in  \cite{LiangZhang, ZhaoSongRepovs}, where the authors did not provide an explicit definition of the operator \eqref{BaldelliBernini:def} itself.

Our main results on the existence of solutions to \eqref{main:equation} are stated below, according to the value of $q$ with respect to $p$. Let us start with the $p$-superlinear, that is, when $q \in (p,p^*_s)$.

\begin{theorem}\label{main:theorem:q>p}
Let $0<s<1<p$, $sp<3$, $p<q<p^*_s$, $A:\R^3 \to \R^3$ be a vector field with locally bounded gradient, $\lambda>0$, and \eqref{H:hyp}, and \eqref{K:hyp} hold. 
Moreover, we assume that \begin{equation}\label{H2:hyp}
\text{there exists} \quad\Omega_H\subset\R^3\quad \text{such that} \quad|\Omega_H|>0\quad \text{and} \quad H>0\quad \text{in} \quad\Omega_H.
\end{equation}

Then, there exists $\lambda^*>0$ such that for any $\lambda>\lambda^*$, then \eqref{main:equation} admits at least one nontriavial solution with positive energy.
\end{theorem}

Of course, condition \eqref{H2:hyp} can be recovered by assuming the continuity of $H$, but we prefer to assume this sharper assumption.

In the setting of Theorem \ref{main:theorem:q>p}, the associated functional enjoys the mountain-pass geometry (cf. Lemma \ref{mpt_geometry}). Therefore, the mountain-pass theorem yields the existence of a Palais-Smale sequence (or shortly, $(PS)-$sequence) that is bounded. Actually, we are able to prove that $(PS)-$sequences are bounded whatever the range of $q$ and $\lambda$ (cf. Lemma \ref{lembound}) and for any level $c$. Next, we prove compactness of such $(PS)-$sequences for suitable ranges of $c$, Lemma \ref{lem5}, overcoming the loss of compactness due to the presence of the critical fractional Sobolev exponent and the entire Euclidean space.

We remark that while the compactness property contained in Lemma \ref{lem5} could be established using the non-magnetic  version of the concentration-compactness principle in conjunction with the  diamagnetic inequality \eqref{diamagnetic:D}, as done in \cite{bfk}, we opt to provide a proof of Lemma \ref{lem5} by of the direct magnetic version, Theorem \ref{Our:concentration:compactness}, as we think it offers a more natural framework for this setting.

While in the previous steps $q$ could coincide with $p$, in the final ingredient to prove the existence of solutions, we restrict ourselves to the case $q \in (p, p_s^*)$, due to technical difficulties arising from the magnetic setting (see Remark \ref{q=pprob}). This last step consists of estimating the mountain pass level and showing that it lies strictly below the threshold at which the $(PS)-$condition holds (see Lemma \ref{csegnato}).

Passing to the $p$-sublinear case, we reach the multiplicity result below.

\begin{theorem}\label{main:theorem:q<p}
Let $0<s<1<p$, $sp<3$, $1<q<p$ and $A:\R^3 \to \R^3$ be a vector field with locally bounded gradient, $\lambda>0$, and \eqref{H:hyp}, \eqref{K:hyp} hold.  Assume also \eqref{H2:hyp}.
Then, there exist $\lambda_*>0$ such that for any $\lambda<\lambda_*$ then \eqref{main:equation} admits a sequence of nontrivial solution with negative energy.
\end{theorem}

The proof of Theorem \ref{main:theorem:q<p} relies on the theory of the Krasnosel’skii genus applied to a truncated functional. This functional is bounded from below and satisfies both the necessary geometric properties (cf. Lemma \ref{genus_geometry}) and the Palais-Smale $(PS)-$condition (cf. Lemma \ref{compeinf}) for any negative level $c$.

Note that Theorem \ref{main:theorem:q<p} is new also in the case $p=2$. Indeed, as far as we know, there are different results in literature dealing with PDEs involving \eqref{dAveniaSquassina:def:2} with variational techniques, with the mountain pass argument as the main tool, also in the semiclassical setting \cite{Ambrosio_schrodinger, Ambrosio_kirchhoff}, but almost no one gets a multiplicity result for solutions with negative energy.

The outline of the paper is as follows. In Section \ref{Section:funct:sett}, we introduce the functional setting. To facilitate readability, the exposition is split into three subsections covering the definition of the space as well as key properties, including embeddings, cut-off estimates, and gauge invariance. Section \ref{Sect:BN} addresses the study of the problem \eqref{main:equation}. After presenting some preliminary results, we investigate the geometry of the energy functional, establishing the boundedness of Palais-Smale sequences and the validity of the Palais-Smale condition as a compactness property. Consequently, we provide the proofs of our main theorems. Finally, the proof of the concentration-compactness result, Theorem \ref{Our:concentration:compactness}, is detailed in Appendix \ref{Appendix}.

\section{Functional setting}
\label{Section:funct:sett}

\subsection{Notations}

We start with setting some basic notation. We indicate with $B_r(x)$ the open $\R^3 $-ball of centered $x\in\R^3$ and radius $r>0$, omitting $x$ when it is the origin, while $B_r^c(x)$ stands for its complement.
Let $C_c^k(\R^3,\R)$ with $k\in[1,\infty]$ be the space of functions with compact support having continuous $k$-th derivative (if $k=1$ we just write $C_c(\R^3,\R)$) and $C_b(\R^3,\R)$ be the space of bounded functions.

Let $M(\R^3,\R)$ be the space of all finite signed Radon measures. Concerning the convergence of sequences of measures $(\mu_n)_n\subseteq M(\R^3,\R)$, we write $\mu_n\stackrel{*}{\rightharpoonup}\mu$ and $\mu_n\rightharpoonup\mu$ for tight and weak convergence, respectively, the weak star convergence with respect to $(C_b(\R^3,\R))'$ and $(C_c(\R^3,\R))'$ (see \cite{BBF} for more details).

For any complex number $a\in\CC^3$  we use the notation 
$a=\Re(a)+i\Im(a)$ with $\Re(a),\Im(a)\in \R^3$.
Let the scalar product in $\R^3$ between $u,v\in\R^3$ be $u\cdot v\in\R$, while the complex product in $\CC^3$ between $u,v\in\CC^3$ is denoted by $\langle u,v\rangle_{\CC^3}=u \cdot \overline{v}\in\CC$.  

Given any measurable set $\Omega\subseteq \R^3$ and $q\in[1,+\infty]$, $L^q(\Omega, \R)$ and $L^q(\Omega, \CC)$ stand for the standard Lebesgue space with real and complex valued, whose norm will be indicated with $\|\cdot\|_{L^q(\Omega, \R)}$ and $\|\cdot\|_{L^q(\Omega, \CC)}$, respectively. Note that $\|\cdot\|_{L^q(\Omega, \CC)}=\||\cdot|\|_{L^q(\Omega, \R)}$.
We recall the following standard inequality.
\begin{equation}\label{standard}
(a+b)^r\le c_r(a^r+b^r), \qquad \text{with}\quad c_r=\max\{2^{r-1},1\}\qquad \text{for all}\quad a,b>0,\,\, r\ge1.
\end{equation}

In what follows, $C$ will denote a positive constant which can change its value at each passage, and $\lesssim$ consists of the inequality $\le$ up to positive constants.

\subsection{The non-magnetic case}\label{nonmagncase}

The present subsection is devoted to recalling some known results concerning the non-magnetic case, in order to fix the notation and facilitate the comparison with the magnetic one.
For any $0<s<1\le p$ and $sp<3$, we define the normalized Gagliardo–Slobodeckiı seminorm
by
$$[u]_{s,p}:=\left(\int_{\R^3\times \R^3}\frac{|u(x) -  u(y)|^p}{|x-y|^{3+ps}} \dd x \dd y \right)^{\frac1p}$$
so that 
$$W^{s,p}(\R^3,\R):=\{u\in L^p(\R^3,\R) : [u]_{s,p}<\infty\}.$$
is a Banach space endowed with the norm
$$\|u\|_{s,p} = \|u\|_{L^p(\R^3,\C)} + [u]_{s,p}.$$
From \cite[Theorem 2.4]{DiPaVa} we can see that $W^{s,p}(\R^3,\R)=W^{s,p}_0(\R^3,\R)$, where $W^{s,p}_0(\R^3,\R)$ is the clousure of $C_c^\infty(\R^3,\R)$ with respect to the norm $\|\cdot\|_{s,p}$.

Note that $W^{s,p}(\R^3,\R)$ and its dual $W^{-s,p'}(\R^3,\R)$ areseparable if $1\le p<\infty$ and reflexive if $1<p<\infty$.

Let us define the homogeneous Sobolev space
$$D^{s,p}_0(\R^3,\R):=\text{completion of $C_c^\infty(\R^3,\R)$ with respect to $[u]_{s,p}$}.$$
which is a Banach space endowed with the norm $\|\cdot\|_{D^{s,p}_0(\R^3,\R)}:=[u]_{s,p}$. Moreover, by \cite[Theorem 3.1]{BrascoGV} $D^{s,p}_0(\R^3,\R)$ can be identified with $D^{s,p}(\R^3,\R)$, defined as
$$D^{s,p}(\R^3,\R):=\{u\in L^{p^*_s}(\R^3,\R) : [u]_{s,p}<\infty\}$$
which is a Banach space with the norm $[\cdot]_{s,p}$, having $C^\infty_c(\R^3,\R)$ has a dense subspace.
Note that $D^{s,p}(\R^3,\R)$ is reflexive if $1<p<\infty$ and separable if $1\le p<\infty$ arguing as for $W^{s,p}(\R^3,\R)$.

Sobolev's Theorem ensures that $D^{s,p}(\R^3,\R) \hookrightarrow L^{p^*_s}(\R^3,\R)$ continuously and the best (embedding) constant $c$ in the Sobolev inequality 
$    \|u\|_{L^{p^*_s}(\R^3,\R)}\leq c [u]_{s,p}$
is given by $S^{-1/p}$, where
\begin{equation}\label{S}
S= \inf_{u\in D^{s,p}(\R^3,\R)\setminus\{0\}} \frac{[u]_{s,p}^p}{\|u\|_{L^{p^*_s}(\R^3,\R)}^p},
\end{equation}
see \cite[Remark 3.3]{BrascoGV}.

\subsection{The non-homogeneous magnetic case}
Assume $0<s<1\le p$ and $sp<3$. Let $A:\R^3 \to \R^3$ be a continuous vector field with locally bounded gradient.
Following \cite{NguyenPinamontiSquassinaVecchi_2018}, we consider the \textit{p-fractional magnetic seminorm} defined as
\begin{equation}\label{normAsp}
    [u]_{A,s,p}:=\left(\int_{\R^3 \times \R^3}\frac{|\expAminus u(x) - u(y)|^p}{|x-y|^{3+ps}} \dd x \dd y \right)^{\frac1p},
\end{equation}
where $u:\R^3\to\CC$. Define the norm
\begin{equation}
\label{complete:normAsp}
    \|u\|_{A,s,p}:=(\|u\|^p_{L^p(\R^3,\C)}+[u]^p_{A,s,p})^{1/p}
\end{equation}
and the space
$$\cB =\left\{u:  \R^3\to\CC : \|u\|_{A,s,p} < +\infty\right\}$$

The following two results are a straightforward generalization to the case $p\neq 2$ arguing as \cite[Proposition 2.1, Proposition 2.2]{dAveniaSquassina}, see also \cite{BerninidAvenia}.
\begin{proposition}\label{prop1}
$\cB$ is a (real) Banach space.
\end{proposition}

\begin{proposition}\label{prop2}
$C^\infty_c(\R^3,\CC)$ is a subspace of $\cB$.
\end{proposition}

Now define
$$W^{s,p}_{A,0}(\R^3,\CC):=\text{closure $C_c^\infty(\R^3,\CC)$ with respect to $\|u\|_{A,s,p}$}.$$
By Proposition \ref{prop2}, then $\cB\supseteq W^{s,p}_{A,0}(\R^3,\CC)$, then by Proposition \ref{prop1}, also $W^{s,p}_{A,0}(\R^3,\CC)$ is a real Banach space. In order to identify the two spaces, it is worth proving the following.

\begin{theorem}
\label{density:theorem}
$C^\infty_c(\R^3,\CC)$ is dense in $\cB$.    
\end{theorem}

For its proof, we proceed as in Step 1 in the proof of \cite[Theorem 3.1]{BrascoGV}. The first step takes into account  the Friedrichs mollifiers, which we recall for the sake of readilibity. Let $\rho \in C^\infty_c(\R^3,\R)$ be such that
\begin{gather}
    \label{mollifiers:definition}
    \supp \rho = B_1, \quad \rho_n=n^3\rho(nx), \text{ for } n \in \NN\setminus\{0\}, \quad \supp \rho_n = B_{\frac1n}\\
    \label{mollifiers:equal:1}
    \int_{\R^3} \rho_n = 1, \text{ for every } n \geq 1.
\end{gather}

\begin{lemma}
\label{Friedrichs:mollifier:lemma}
    For every $u \in L^1_\loc(\R^3,\CC)$ such that
    \begin{equation}\label{hplemma}
        [u]_{A,s,p} < +\infty
    \end{equation}
    then, it holds
    \begin{equation}
        \lim_{n \to +\infty}  [(u*\rho_n)-u]_{A,s,p}=0.
    \end{equation}
\end{lemma}

\begin{proof}
    By using \eqref{normAsp} and \eqref{mollifiers:equal:1}, we have
{\small \begin{equation}\label{eqmollsem}\begin{aligned}
&[u\ast \rho_n-u]_{A,s,p}^p:=\int_{\R^3 \times \R^3}\frac{|\expAminus (u\ast \rho_n-u)(x) -  (u\ast \rho_n-u)(y)|^p}{|x-y|^{3+ps}} \dd x \dd y \\
&=\int_{\R^3 \times \R^3}\frac{|\int_{\R^3}\expAminus(u(x-z)-u(x))\rho_n(z) \dd z -  \int_{\R^3}(u(y-w)-u(y))\rho_n(w) \dd w|^p}{|x-y|^{3+ps}} \dd x \dd y \\
&\le \int_{\R^3 \times \R^3}\frac{\int_{\R^3 \times \R^3}|\expAminus(u(x-z)-u(x)) -  (u(y-w)-u(y))|^p\rho_n(z)\rho_n(w) \dd z \dd w}{|x-y|^{3+ps}} \dd x \dd y 
\end{aligned}\end{equation}}
where in the last inequality we use the Jensen's inequality.
By setting
\begin{equation}
    F(x,y):=\frac{\expAminus u(x)- u(y)}{|x-y|^{\frac3p+s}}
\end{equation}
we observe that \eqref{hplemma} implies $F \in L^p(\R^3 \times \R^3,\CC)$ and 
\begin{equation}
    \begin{aligned}
        &|F(\cdot-z,\cdot-w)-F(x,y)|^p\\
        &= 
        \left|\frac{\e^{-i(x-z-y+w)\cdot A\left(\frac{x-z+y-w}{2} \right)}u(x-z)- u(y-w)}{|x-z-y+w|^{\frac3p + s}}- \frac{\expAminus u(x)- u(y)}{|x-y|^{\frac3p + s}}\right|^p\\
        &\lesssim
        \left|\frac{\e^{-i(x-z-y+w)\cdot A\left(\frac{x-z+y-w}{2} \right)}u(x-z)- u(y-w)}{|x-z-y+w|^{\frac3p + s}}\right|^p
        +
        \left|\frac{\expAminus u(x)- u(y)}{|x-y|^{\frac3p + s}}\right|^p \in L^1(\R^3 \times \R^3,\CC)
    \end{aligned}
\end{equation}
thanks to \eqref{hplemma} and \eqref{standard}. So, by the Dominated Convergence Theorem, it follows that
\begin{equation}
\label{numerator:to:zero}
\begin{aligned}
    &\left\|F(\cdot-z,\cdot-w)-F(x,y) \right\|_{L^p(\R^3,\C)}\to 0
\end{aligned}\end{equation}
as $|z|,|w| \to 0$.
Hence, by \eqref{mollifiers:equal:1} and \eqref{numerator:to:zero}, we can change the order of integration in the latter line of \eqref{eqmollsem}, and we conclude by letting $n \to +\infty$.
\end{proof}

Further, the following cut-off truncation, generalized to the magnetic setting, is needed.

\begin{lemma}
\label{cut:off:truncation:lemma}
Let $u \in \cB$. If $(\varphi_n)_n\subset C_c^\infty(\R^3,\R)$ is a sequence of cut-off functions such that $0\le \varphi_n\le 1$ and
\begin{equation}
	\label{phi_n:definition}
	\varphi_n\equiv 0\quad\text{on} \, B_n, \qquad \varphi_n\equiv 1\quad\text{on} \, B_{2n}^c, \qquad \|\nabla \varphi_n\|_{L^\infty(\R^3)}\le \frac{C}{n},
\end{equation}
where $C>0$, then we have
$$\|\varphi_n u\|_{L^p(\R^3,\C)}+[\varphi_n u]_{A,s,p}\to 0\qquad \text{as}\quad n\to\infty.$$  
\end{lemma}

\begin{proof}
Concerning the Lebesgue's norm, we have

$$\|\varphi_n u\|_{L^p(\R^3,\C)}^p=\int_{\R^3}|\varphi_n(x)|^p |u(x)|^p \dd x\le \int_{B_{n}^c} |u(x)|^p \dd x\to 0$$
by the Dominated Convergence Theorem.

Now, we decompose the seminorm as follows 
\begin{equation}
\begin{aligned}
    [\varphi_n u]_{A,s,p}
    &=
    \int_{\R^3 \times \R^3} \frac{|\expAminus\varphi_n(x)u(x)-\varphi_n(y)u(y)|^p}{|x-y|^{3+2s}}\dd x \dd y \\
    &=
    2J_1 + 2J_2 + J_3 +2J_4 + J_5,
\end{aligned}
\end{equation}
where
\begin{align}
       &J_1=\int_{B_n \times A_n}\frac{|\varphi_n(y)u(y)|^p}{|x-y|^{3+sp}}\dd x \dd y \label{J1}\tag{{\it $J_1$}}\\
       &J_2=\int_{B_n \times B_{2n}^c}\frac{|u(y)|^p}{|x-y|^{3+sp}}\dd x \dd y\label{J2}\tag{{\it $J_2$}}\\
       &J_3=\int_{A_n \times A_n}\frac{|\expAminus\varphi_n(x)u(x)-\varphi_n(y)u(y)|^p}{|x-y|^{3+sp}}\dd x \dd y\label{J3}\tag{{\it $J_3$}}\\
       &J_4=\int_{B_{2n}^c \times A_n}\frac{|\expAminus u(x)-\varphi_n(y)u(y)|^p}{|x-y|^{3+sp}}\dd x \dd y\label{J4}\tag{{\it $J_4$}}\\
       &J_5=\int_{B_{2n}^c \times B_{2n}^c}\frac{|\expAminus u(x)- u(y)|^p}{|x-y|^{3+sp}}\dd x \dd y,\label{J5}\tag{{\it $J_5$}}
\end{align}
where $A_n:=B_{2n}\setminus B_n$. We proceed into steps.

{\bf Estimate of \eqref{J1}.}
Note that, by using the Mean Value Theorem and \eqref{phi_n:definition}, we have
$$\begin{aligned}
J_1&=\int_{B_n \times A_n}\frac{|\varphi_n(y)u(y)|^p}{|x-y|^{3+sp}}\dd x \dd y=\int_{B_n \times A_n}\frac{|\varphi_n(y)-\varphi_n(x)|^p|u(y)|^p}{|x-y|^{3+sp}}\dd x \dd y\\
&\le \frac{C}{n^p}\int_{B_n \times A_n}\frac{|u(y)|^p}{|x-y|^{3+sp-p}}\dd x \dd y.
\end{aligned}$$
Note that, for any $y\in A_n$, then $B_n\subset B_{3n}(y)$, so that
$$\int_{B_n}\frac{1}{|x-y|^{3+sp-p}}\dd x \le \int_{B_{3n}(y)}\frac{1}{|x-y|^{3+sp-p}}\dd x 
=\omega_2\int_0^{3n}\rho^{-sp+p-1}d\rho
=Cn^{p(1-s)},$$
implying
$$J_1\le \frac{C}{n^{sp}} \int_{A_n} |u(y)|^p \dd x.$$
So $J_1\to0$ as $n\to\infty$ by using that $u\in L^p(\R^3,\CC)$ and the Dominated Convergence Theorem.

{\bf Estimate of \eqref{J2}.}
Since $B_n(y) \subset B_n^c$ for all $y\in B_{2n}^c$, then $B_n \subset B^c_n(y)$, hence
\begin{equation}
    \int_{B_n}\frac{1}{|x-y|^{3+sp}}\dd x
    \leq 
    \int_{B^c_n(y)}\frac{1}{|x-y|^{3+sp}}\dd x
    =
    \omega_2\int_n^{+\infty} \frac{1}{\rho^{1+sp}} \dd \rho \to 0
\end{equation}
as $n$ diverges. Therefore,
\begin{equation}
    \begin{aligned}
        J_2
        =
        \int_{B_n \times B_{2n}^c}\frac{|u(y)|^p}{|x-y|^{3+sp}}\dd x \dd y 
        =
        \int_{B_{2n}^c}|u(y)|^p \int_ {B_n}\frac{1}{|x-y|^{3+sp}}\dd x \dd y
        \leq 
        \omega_2\int_{B_{2n}^c}|u(y)|^p \int_n^{+\infty} \frac{1}{\rho^{1+sp}} \dd \rho \to 0
    \end{aligned}
\end{equation}
as $n$ goes to infinity, since $u \in L^p(\R^3,\CC)$.

{\bf Estimate of \eqref{J3}.}
In this case, by using \eqref{phi_n:definition}, the standard inequality \eqref{standard} and the Mean Value Theorem, after adding and substracting $\expAminus\varphi_n(y)u(x)$, we compute
\begin{equation}
    \begin{aligned}
        J_3
        &=
        \int_{A_n \times A_n}\frac{|\expAminus\varphi_n(x)u(x)-\varphi_n(y)u(y)|^p}{|x-y|^{3+sp}}\dd x \dd y \\
        & \lesssim 
        \int_{A_n \times A_n}\frac{|\varphi_n(x)-\varphi_n(y)|^p |u(x)|^p}{|x-y|^{3+sp}}\dd x \dd y
        + 
        \int_{A_n \times A_n}\frac{|\expAminus u(x)- u(y)|^p|\varphi_n(y)|^p}{|x-y|^{3+sp}}\dd x \dd y \\
        & \lesssim 
        \frac{C}{n^p}\int_{A_n \times A_n}\frac{ |u(x)|^p}{|x-y|^{3+sp-p}}\dd x \dd y
        + 
        \int_{A_n \times A_n}\frac{|\expAminus u(x)- u(y)|^p}{|x-y|^{3+sp}}\dd x \dd y \\
        & \lesssim 
        \frac{C}{n^p}\int_{A_n}|u(x)|^p\int_{A_n}\frac{ 1}{|x-y|^{3+sp-p}}\dd y \dd x
        + 
        \int_{A_n \times A_n}\frac{|\expAminus u(x)- u(y)|^p}{|x-y|^{3+sp}}\dd x \dd y.
    \end{aligned}
\end{equation}
The first integral goes to $0$ as for \eqref{J1}, while for the second one we argue by means of the Dominated Convergence Theorem.

{\bf Estimate of \eqref{J4}.}  
By using \eqref{standard}, adding and subtracting $\varphi_n(x)u(y)(=u(y))$ in the numerator of \eqref{J4}, we have
\begin{equation}\begin{aligned}
    J_4&=\int_{B_{2n}^c \times A_n}\frac{|\expAminus u(x)-\varphi_n(y)u(y)|^p}{|x-y|^{3+sp}}\dd x \dd y\\
    &\lesssim\int_{B_{2n}^c \times A_n}\left(\frac{|\expAminus u(x)- u(y)|^p}{|x-y|^{3+sp}}+\frac{|\varphi_n(x)-\varphi_n(y)|^p|u(y)|^p}{|x-y|^{3+sp}}\right)\dd x \dd y:=J_{4,1}+J_{4,2}.
\end{aligned}\end{equation}
By applying the Dominated Convergence Theorem, one can prove that $J_{4,1}\to 0$ as $n\to\infty$. Concerning $J_{4,2}$, we decompose it as follows
\begin{equation}\begin{aligned}
J_{4,2}&=\int_{B_{2n}^c \times A_n}\frac{|\varphi_n(x)-\varphi_n(y)|^p|u(y)|^p}{|x-y|^{3+sp}}\dd x \dd y\\
&=\int_{ B_{3n}^c\times A_n}\frac{|\varphi_n(x)-\varphi_n(y)|^p|u(y)|^p}{|x-y|^{3+sp}}\dd x \dd y+\int_{  (B_{3n}\setminus B_{2n})\times A_n}\frac{|\varphi_n(x)-\varphi_n(y)|^p|u(y)|^p}{|x-y|^{3+sp}}\dd x \dd y\\
&\le C\int_{ B_{3n}^c\times A_n}\frac{|u(y)|^p}{|x-y|^{3+sp}}\dd x \dd y+\frac{C}{n^p}\int_{  (B_{3n}\setminus B_{2n})\times A_n}\frac{|u(y)|^p}{|x-y|^{3+sp-p}}\dd x \dd y
\end{aligned}\end{equation}
where in the last inequality we both used the fact that $\varphi_n\le 1$ and the Mean value Theorem. For the second integral we can proceed as for \eqref{J1}, while for the first one, we observe that for any $y\in A_n$, then $B_n(y)\subset B_{3n}$, giving $B_{3n}^c\subset B_{n}^c(y)$, and so, arguing again as in \eqref{J1}, as $n$ goes to $+\infty$ we get
$$\int_{ B_{3n}^c\times A_n}\frac{|u(y)|^p}{|x-y|^{3+sp-p}}\dd x \dd y
\leq
\frac{C}{n^{sp}}\int_{A_n}|u(y)|^p \dd x \to 0.$$

{\bf Esitmate of \eqref{J5}.}
Clearly, since
\begin{equation}
    \frac{|\expAminus u(x) - u(y)|}{|x-y|^{\frac3p+s}} \in L^p(\R^3 \times \R^3) 
\end{equation}
by the Dominated Convergence Theorem, \eqref{J5} goes to $0$ as $n\to\infty$.
\end{proof}

We are finally in the position to prove Theorem \ref{density:theorem}
\begin{proof}[Proof of Theorem \ref{density:theorem}.]
We need to prove that for every $u\in \cB$, there exists a sequence 
$(u_n)_n\subset C^\infty_c(\R^3,\C)$ such that
$$\|u_n-u\|_{L^p(\R^3,\C)}+[u_n-u]_{A,s,p}\to 0\quad \text{as} \,\, n\to\infty.$$
Consider the sequence of mollifiers in \eqref{mollifiers:definition}, and take a sequence of cut-off functions $(\eta_j)_j \subset C^\infty_c(\R^3,\R)$ such that
\begin{equation}\label{etaj}
    \supp \eta_j \subset B_{2j}, \quad 0 \leq \eta_j \leq 1, \quad \eta_j \equiv 1 \text{ on } B_j, \quad \|\nabla \eta_j\|_{L^\infty(\R^3,\C)} \leq \frac Cj.
\end{equation}
Fixed $n \geq 1$, by Lemma \ref{cut:off:truncation:lemma} (with $|\eta_j-1|=\varphi_j$ and $(u*\rho_n)\in\cB$), we get
\begin{equation}
    \| (u*\rho_n)\varphi_j\|_{L^p(\R^3,\C)} +[(u*\rho_n)\varphi_j]_{A,s,p}=\|(u*\rho_n)\eta_j - (u * \rho_n)\|_{L^p(\R^3,\C)} +  [(u*\rho_n)\eta_j - (u * \rho_n)]_{A,s,p}
     \to 0,
\end{equation}
as $j$ diverges. Namely, for every $n \geq 1$ there exists $j_n \in \NN$ such that for all $j\ge j_n$
\begin{equation}
    \label{density:triangle:in:1}
    \|(u*\rho_n)\eta_j - (u * \rho_n)\|_{L^p(\R^3,\C)} + [(u*\rho_n)\eta_j - (u * \rho_n)]_{A,s,p} \leq \frac1n.
\end{equation}
We now set     $u_n=(u*\rho_n)\eta_j.$
On one hand,
\begin{equation}
\begin{aligned}
    [u_n-u]_{A,s,p}
    &=
    [u_n - (u*\rho_n) + (u*\rho_n) - u]_{A,s,p}\\
    &\leq
    [u_n - (u*\rho_n)]_{A,s,p} + [(u*\rho_n) - u]_{A,s,p}
    \leq
    \frac1n + [(u*\rho_n) - u]_{A,s,p} \to 0
\end{aligned}
\end{equation}
as $n$ diverges to $+\infty$, thanks to \eqref{density:triangle:in:1} and Lemma \ref{Friedrichs:mollifier:lemma}.
On the other hand, we have
\begin{equation}
\begin{aligned}
    \|u_n - u\|_{L^p(\R^3,\C)} 
    &= 
    \|(u*\rho_n)\eta_j - u\|_{L^p(\R^3,\C)}\\
    &\leq
    \|(u*\rho_n)\eta_j - (u*\rho_n)\|_{L^p(\R^3,\C)}
    +
    \|(u*\rho_n) - u\|_{L^p(\R^3,\C)} \to 0 
\end{aligned}
\end{equation}
thanks to \eqref{density:triangle:in:1} and \cite[Theorem 4.22]{Brezis}.
The proof is completed.
\end{proof}
Hence, we have proven that
    $\cB=W^{s,p}_{A,0}(\R^3,\CC).$
For simplicity, we set
$W^{s,p}_A(\R^3,\CC):=\cB$.

The following lemma provides a fundamental estimate relating the norm of the magnetic Sobolev space to that of the standard non-homogeneous Sobolev space.

\begin{lemma}[Diamagnetic inequality]\label{lem:diam}
    For every $u \in W^{s,p}_A(\R^3,\CC)$ it holds 
    \begin{equation}
        \||u|\|_{s,p} \leq \|u\|_{A,s,p},
    \end{equation}
    and so, $|u| \in W^{s,p}(\R^3,\R)$.
\end{lemma}
\begin{proof}
It easily follows from the pointwise diamagnetic inequality
$$||u(x)|-|u(y)||\le |\expAminus u(x)- u(y)|\qquad \text{for a.e. $x,y\in \R^3$},$$
see \cite[Remark 3.2]{dAveniaSquassina}.
\end{proof}

In the next Lemmas we show some embedding results, inspired by the semilinear case in \cite{dAveniaSquassina}.

\begin{lemma}[Local embedding]
\label{local:cont:emb}
$W^{s,p}_A(\R^3,\C) \hookrightarrow W^{s,p}_{\loc}(\R^3,\C)$, where
$$W^{s,p}_\loc(\R^3,\C):=\{u:\R^3\to\C : u\in W^{s,p}(\Omega,\C) \text{ for all open } \Omega \subset\subset \R^3\},$$
endowed with norm \eqref{complete:normAsp}, taken with $A \equiv 0$.
\end{lemma}
\begin{proof}
Let $\Omega$ be an open set such that $\Omega \subset\subset \R^3$. Then, adding and substracting, by \eqref{standard}, we get
\begin{align*}
	&\lesssim \int_{\Omega}|u(x)|^p \dd x
	+ \int_{\Omega \times \Omega}\frac{\left|\e^{-i(x-y)\cdot A\left(\frac{x+y}{2}\right)}u(x) - u(y)\right|^p}{|x-y|^{3+ps}} \dd x \dd y \\
	&\qquad + \int_{\Omega \times \Omega}\frac{\left|u(x)\left(\e^{-i(x-y)\cdot A\left(\frac{x+y}{2}\right)}-1\right)\right|^p}{|x-y|^{3+ps}} \dd x \dd y \\
	& \lesssim \|u\|_{A,s,p}^p + \int_{\Omega \times \Omega}\frac{\left|u(x)\left(\e^{-i(x-y)\cdot A\left(\frac{x+y}{2}\right)}-1\right)\right|^p}{|x-y|^{3+ps}} \dd x \dd y.
\end{align*}
Moreover, since $|\e^{i\theta} - 1| \leq 2$ for every $\theta \in \R$, and $|\e^{i\theta} - 1| \leq C|\theta|$ for small $\theta \in \R$, we have
\begin{align*}
	&    \int_{\Omega \times \Omega}\frac{\left|u(x)\left(\e^{-i(x-y)\cdot A\left(\frac{x+y}{2}\right)}-1\right)\right|^p}{|x-y|^{3+ps}}\dd x \dd y \\
    &  \quad =\int_{\Omega}|u(x)|^p\left( \int_{\Omega \cap \left\{|x-y| \geq 1\right\}}\frac{\left|\e^{-i(x-y)\cdot A\left(\frac{x+y}{2}\right)}-1\right|^p}{|x-y|^{3+ps}} \dd y\right) \dd x \\
    & \quad\qquad + \int_{\Omega}|u(x)|^p \left( \int_{\Omega \cap \left\{|x-y| < 1\right\}}\frac{\left|\e^{-i(x-y)\cdot A\left(\frac{x+y}{2}\right)}-1\right|^p}{|x-y|^{3+ps}} \dd y\right)  \dd x\\
	& \quad\leq C\left(\int_{\Omega}|u(x)|^p \left(\int_{\Omega \cap \left\{|x-y| \geq 1\right\}}\frac{1}{|x-y|^{3+ps}} \dd y \right.\right.\\
	& \quad\qquad \left.\left.+ \|A\|^p_{L^{\infty}(\Omega, \CC)}\int_{\Omega}|u(x)|^p \dd x \int_{\Omega \cap \left\{|x-y| < 1\right\}}\frac{1}{|x-y|^{3+ps-p}} \dd y \right)\right) \dd x \\
	& \quad\leq C(s,p,\|A\|_{L^{\infty}(\Omega, \CC)})\int_{\Omega}|u(x)|^p \dd x \left(\int_1^{+\infty}\frac{1}{r^{1+ps}} \dd r + \int_0^1\frac{1}{r^{1+ps-p}} \dd r \right)\\
	& \quad\leq C(s,p,\|A\|_{L^{\infty}(\Omega, \CC)})\int_{\Omega}|u(x)|^p \dd x.
\end{align*}
Therefore, we can conclude.
\end{proof}

\begin{lemma}[Sobolev embeddings]
\label{Sobolev:emedding:lemma}
The embedding
    \begin{equation}
        W^{s,p}_A(\R^3,\CC) \hookrightarrow L^q(\R^3,\CC),
    \end{equation}
    is continuous for every $q \in [p,p^*_s]$, while the embedding
    \begin{equation}
        W^{s,p}_A(\R^3,\CC) \hookrightarrow L^q_\loc(\R^3,\CC),
    \end{equation}
   is compact for every $q \in [1,p^*_s)$.
\end{lemma}
\begin{proof}
    Let $u \in W^{s,p}_A(\R^3,\CC)$, then by recalling $\|\cdot\|_{L^q(\R^3,\CC)}=\| |\cdot| \|_{L^q(\R^3,\R)}$, using \cite[Theorem 6.5]{DiPaVa} and Lemma \ref{lem:diam}, for any $q \in [p,p^*_s]$, we have
    \begin{equation}
        \begin{aligned}
            \|u\|_{L^q(\R^3,\CC)}
            &=
            \||u|\|_{L^q(\R^,\R) }
            \leq
            C\left(\int_{\R^3\times \R^3} \frac{||u(x)|-|u(y)||^p}{|x-y|^{3+sp}} \dd x \dd y\right)^\frac1p\\
            &\leq
            C\left(\int_{\R^3\times \R^3} \frac{|\expAminus u(x)-u(y)|^p}{|x-y|^{3+sp}} \dd x \dd y\right)^\frac1p=
            C[u]_{A,s,p}
            \leq
            \|u\|_{A,s,p}.
        \end{aligned}
    \end{equation}
About the compactness of the embedding, it is enough to use Lemma \ref{local:cont:emb} and \cite[Corollary 7.2]{DiPaVa}.
\end{proof}

As in the Hilbert setting, arguing as in \cite[Lemma 3.10 and Lemma 3.11]{dAveniaSquassina} the following gauge invariance can be shown.
\begin{lemma}
    Let $\xi, \eta \in \R^3$, and $u \in W^{s,p}_A(\R^3,\C)$. If 
\[
    v(x):=\e^{i\eta\cdot x}u(x+\xi), \quad A_{\eta}(x):=A(x+\xi) + \eta, \quad x \in \R^3,
\]
then, $v \in W^{s,p}_{A_\eta}(\R^3,\C)$ and $[u]_{A,s,p} = [v]_{A_\eta,s,p}$.
If $A$ is linear and $\eta=-A(\xi)$, then $[u]_{A,s,p} = [v]_{A,s,p}$.
\end{lemma}

\subsection{The homogeneous magnetic case}\label{hom:magn}

Let $0<s<1<p$, $sp<3$ and $A:\R^3 \to \R^3$ be a continuous vector field with locally bounded gradient.
Following \cite{NguyenPinamontiSquassinaVecchi_2018}, we consider the \textit{p-fractional magnetic seminorm} defined in \eqref{normAsp}, here recalled for simplicity
\begin{equation}\label{normAsp:D}
    [u]_{A,s,p}:=\left(\int_{\R^3\times \R^3}\frac{|\expAminus u(x) - u(y)|^p}{|x-y|^{3+ps}} \dd x \dd y \right)^{\frac1p}.
\end{equation}
As in \cite{BrascoGV}, 
$$
    (C^\infty_c(\R^3,\CC), [\cdot]_{A,s,p})\qquad \text{is a normed space but not complete,}
$$
differently from the non homogeneous case where 
$(C^\infty_c(\R^3,\CC), \|\cdot\|_{A,s,p})$ was a Banach space.

So, we define
$$D^{s,p}_{A,0}(\R^3,\CC):=\text{completion $C_c^\infty(\R^3,\CC)$ with respect to $[u]_{A,s,p}$}.$$
By definition, $D^{s,p}_{A,0}(\R^3,\CC)$ is the quotient space of the set of sequences
\[
(u_m)_{m \in \mathbb{N}} \subset C_c^\infty(\mathbb{R}^3, \CC)
\]
which are Cauchy for the norm $[\cdot]_{A,s,p}$, under the expected equivalence relation
\[
(u_m)_{m \in \mathbb{N}} \sim_{A,s,p} (v_m)_{m \in \mathbb{N}}
\quad \text{if and only if} \quad
\lim_{m \to \infty} [u_m - v_m]_{A,s,p} = 0.
\]
Thus, for each equivalence class $U=\{(u_n)_n\}_{A,s,p}$ we define the norm of $D^{s,p}_{A,0}(\R^3,\CC)$ in terms of a
representative as
    \begin{equation}
    \label{defnormabrutta}
        \|U\|_{D^{s,p}_{A,0}(\R^3,\CC)}:=\lim_{m\to\infty}[u_m]_{A,s,p}.
    \end{equation}
It is easily seen that such a definition is independent of the representative.
With this construction, $(D^{s,p}_{A,0}(\R^3,\CC),\|\cdot\|_{D^{s,p}_{A,0}(\R^3,\CC)})$ is a Banach space.

Note that all functions $u \in C^\infty_c(\R^3,\C)$ can be naturally embedded into
$D^{s,p}_{A,0}(\R^3,\C)$ by the constant sequence $u_m = u$. By definition, $C^\infty_c(\R^3,\CC)$ is
dense in $D^{s,p}_{A,0}(\R^3,\C)$ with respect to $\|\cdot\|_{D^{s,p}_{A,0}(\R^3,\CC)}$.

Now, we define the space
$$
\cD :=\left\{u :\R^3 \to \C : [u]_{A,s,p} < +\infty\right\} = \left\{u \in L^{p^*_s}(\R^3,\C) : [u]_{A,s,p} < +\infty\right\},
$$
thanks to the following \textit{magnetic Sobolev inequality}
\begin{equation}
    \label{magnetic:Sobovel:inequality}
    \|u\|_{L^{p^*_s}(\R^3,\C)} \leq C[u]_{A,s,p} \text{ for every } u \in L^{p^*_s}(\R^3,\C),
\end{equation}
where $C>0$ and which follows from the proof of Lemma \ref{Sobolev:emedding:lemma}, it is possible to give a concrete characterization of the completion
$D_{A,0}^{s,p}(\R^3,\CC)$ as a functional space.\\

We endow $\cD$ with $[\cdot]_{A,s,p}$ and we first prove its completeness.
\begin{proposition}\label{prop1D}
$(\cD, [\cdot]_{A,s,p})$ is a real Banach space.
\end{proposition}
\begin{proof}
It follows by the same arguments as in Proposition \ref{prop1} and by \eqref{magnetic:Sobovel:inequality}.
\end{proof}

\begin{proposition}\label{prop2D}
$C^\infty_c(\R^3,\CC)$ is a subspace of $\cD$.
\end{proposition}
\begin{proof}
The proof follows the same lines as in Proposition \ref{prop2}.
\end{proof}

To complete the characterization we are looking for, we need to prove that $C_c^\infty(\R^3,\CC)$ is dense in $\cD$.

\begin{theorem}
\label{density:theorem:D}
$C^\infty_c(\R^3,\CC)$ is dense in $\cD$ with respect to $[\cdot]_{A,s,p}$.    
\end{theorem}
In order to prove it, we proceed as in Step 1 in the proof of \cite[Theorem 3.1]{BrascoGV}. The first step takes into account  the Friedrichs mollifiers
$(\rho_n)_n \subset C^\infty_c(\R^3,\CC)$ defined as \eqref{mollifiers:definition}.

Further, the following cut-off truncation, taking Lemma \ref{cut:off:truncation:lemma} to the homogeneous magnetic setting, is needed.

\begin{lemma}
\label{cut:off:truncation:lemmaD}
Let $u \in \cD$. If $(\varphi_n)_n\subset C_c^\infty(\R^3,\R)$ is a sequence of cut-off functions such that $0\le \varphi_n\le 1$ and \eqref{phi_n:definition} holds,
then we have
$$[\varphi_n u]_{A,s,p}\to 0\qquad \text{as}\quad n\to\infty$$  
\end{lemma}

We are finally in the position to prove Theorem \ref{density:theorem:D}.

\begin{proof}[Proof of Theorem \ref{density:theorem:D}.]
We need to prove that for every $u\in \cD$, there exists a sequence 
$(u_n)_n\subset C^\infty_c(\R^3,\C)$ such that
$$[u_n-u]_{A,s,p}\to 0\quad \text{as} \,\, n\to\infty.$$
Consider the sequence of mollifiers $(\rho_n)_n$ introduced in \eqref{mollifiers:definition}, and the sequence of cut-off functions $(\eta_j)_j \in C^\infty_c(\R^3,\R)$ such that \eqref{etaj} holds.
Fixed $n \geq 1$, by Lemma
\ref{cut:off:truncation:lemmaD} (with $|\eta_j-1|=\varphi_j$ and $(u*\rho_n)\in\cD$)
\begin{equation}
      [(u*\rho_n)\eta_j - (u * \rho_n)]_{A,s,p}
    = 
     [(u*\rho_n)\varphi_j]_{A,s,p} \to 0,
\end{equation}
as $j$ diverges, that is, for every $n \geq 1$ there exists $j_n \in \NN$ such that for all $j\ge j_n$
\begin{equation}
    \label{density:triangle:in:1:D}
    [(u*\rho_n)\eta_j - (u * \rho_n)]_{A,s,p} \leq \frac1n.
\end{equation}
We now set 
   $ u_n=(u*\rho_n)\eta_j$
and note that 
\begin{equation}
\begin{aligned}
    [u_n-u]_{A,s,p}
    &=
    [u_n - (u*\rho_n) + (u*\rho_n) - u]_{A,s,p}\\
    &\leq
    [u_n - (u*\rho_n)]_{A,s,p} + [(u*\rho_n) - u]_{A,s,p}
    \leq
    \frac1n + [(u*\rho_n) - u]_{A,s,p} \to 0
\end{aligned}
\end{equation}
as $n$ diverges to $+\infty$, thanks to \eqref{density:triangle:in:1:D} and Lemma \ref{Friedrichs:mollifier:lemma}.\\
The proof is completed.
\end{proof}

Since $\cD$ and $D^{s,p}_{A,0}(\R^3,\CC)$ are both Banach spaces but endowed with different norms, we cannot say that they coincide. Nevertheless, as in \cite{BrascoGV}, we can prove an identification between them by an isometric isomorphism.

\begin{proposition}
\label{isometric:isomorphism}
    There exists an isometric isomorphism $j:D^{s,p}_{A,0}(\R^3,\C) \to \cD$. In other words, the space $D^{s,p}_{A,0}(\R^3,\C)$ can be identified with $\cD$.
\end{proposition}
\begin{proof}
    We argue as in \cite[Theorem 3.1]{BrascoGV}.
    Take $U \in D^{s, p}_{A,0}(\R^3,\C)$ and choose a representative of this equivalence class, i.e. $U=\left\{(u_n)_n\right\}_{A,s, p}$. Thanks to Proposition \ref{prop1D} and Proposition \ref{prop2D}, we know that $(u_n)_n$ converges to a function $u \in \cD$. We then define
        $j(U)=u.$
    Observe that this is well-defined since ,for any other representative $(\widetilde{u}_n)_n$ belonging to the class $U$, we still have
    $$
        \lim _{n \rightarrow \infty}\left[\widetilde{u}_n-u\right]_{A,s,p} 
        \leq 
        \lim _{n \rightarrow +\infty}\left[\widetilde{u}_n-u_n\right]_{A,s,p}+\lim _{n \rightarrow \infty}\left[u_n-u\right]_{A,s,p}=0.
    $$
    It is easy to see that $j$ is linear. It is also immediate to obtain that this is an isometry, since by \eqref{defnormabrutta}
    $$
        \|U\|_{D^{s,p}_{A,0}(\R^3)}=\lim _{n \rightarrow +\infty}\left[u_n\right]_{A,s,p}=[u]_{A,s,p}=[j(U)]_{A,s,p}.
    $$
We are left with proving that $j$ is surjective. From Theorem \ref{density:theorem:D} we know that for every $u \in \cD$ there exists a sequence $(u_n)_n \subset C^{\infty}_c(\R^3,\C)$ such that
    $$
        \lim _{n \rightarrow \infty}\left[u_n-u\right]_{A,s,p}=0.
    $$
In particular, $(u_n)_n$ is a Cauchy sequence with respect to the seminorm. Thus we get
    $$
        u=j(\left\{(u_n)_n\right\}_{A,s, p}).
    $$
This concludes the proof.
\end{proof}

For simplicity, we set
$D^{s,p}_A(\R^3,\CC):=\cD$. Note that, if $A=0$, we recover the homogeneous space $D^{s,p}(\R^3,\R)$ studied in \cite{BrascoGV}.

Concerning Sobolev embeddings, we observe that \eqref{magnetic:Sobovel:inequality} yields
\begin{equation}
\label{homogeneous:Sobolev:embedding}
    D^{s,p}_A(\R^3,\C) \hookrightarrow L^{p^*_s}(\R^3,\C),
\end{equation}
where the best embedding constant defined as
\begin{equation}\label{SA}
S_A= \inf_{u\in D^{s,p}_A(\R^3,\CC)\setminus\{0\}} \frac{[u]_{A,s,p}^p}{\|u\|_{L^{p^*_s}(\R^3,\C)}^p}. 
\end{equation}
allowing to rewrite \eqref{magnetic:Sobovel:inequality} as
\begin{equation}\label{magnetic:Sobovel:inequality:Dspa}
    \|u\|_{L^{p^*_s}(\R^3,\C)} \leq S_A^{-1/p}[u]_{A,s,p} \qquad \text{for all}\,\, u\in D^{s,p}_A(\R^3,\C).
\end{equation}
Moreover, note that the diamagnetic inequality (cf. Lemma \ref{lem:diam}) still holds in the homogeneous case, that is for any $u\in D^{s,p}_A(\R^3,\CC)$, i.e.
\begin{equation}\label{diamagnetic:D}
[ |u| ]_{s,p}\le [u]_{A,s,p},
\end{equation}
yelding $|u| \in D^{s,p}(\R^3,\R)$.

To end the present section, inspired by \cite[Lemma 4.6]{dAveniaSquassina} we state following Lemma, which compares the two Sobolev constants, i.e. the nonmagnetic \eqref{S} and the magnetic one \eqref{SA}, raising an important feature. We believe that this result could be of independent interest.
\begin{lemma}
\label{Sobolev:constants:equal}
Let $A:\R^3 \to \R^3$ be a continuous field with locally bounded gradient. Then, $S=S_A$.
\end{lemma}

\begin{proof}
Recalling the definition of \eqref{SA}, we can also define $S_A$, without loss of generality, as 
\begin{equation}\label{SSA}
S_A= \inf_{\substack{u\in D^{s,p}_A(\R^3,\CC)\cap C^\infty_c(\R^3,\C)\\{\|u\|_{L^{p^*_s}(\R^3,\CC)}=1}}} [u]_{A,s,p}^p\qquad \text{and}\qquad S_0= \inf_{\substack{u\in D^{s,p}(\R^3,\CC)\cap C^\infty_c(\R^3,\C)\\{\|u\|_{L^{p^*_s,}(\R^3,\CC)}=1}}} [u]_{s,p}^p . 
\end{equation}
Moreover, since $[|u|]_{s,p}\le [u]_{s,p}$, we have
\begin{equation}\label{SSA:normalized}
S_0= \inf_{\substack{u\in D^{s,p}(\R^3,\R)\cap C^\infty_c(\R^3,\R)\\{\|u\|_{L^{p^*_s,}(\R^3,\CC)}=1}}} [u]_{s,p}^p =S.
\end{equation}

Let $\eps > 0$, take $u \in C^\infty_c(\R^3,\C)$ such that 
\begin{equation}\label{Sepsilon}
       \|u\|_{L^{p^*_s}(\R^3,\CC)}=1, \qquad [u]_{s,p}^p \leq S + \eps,
\end{equation}
by \eqref{SSA}.
Consider the following rescaling
\begin{equation}
    u_\sigma(x)=\sigma^{-\frac{3-ps}{p}}u\left(\frac{x}{\sigma}\right), \quad \sigma>0, \,\,x \in \R^3.
\end{equation}
Clearly,
\begin{equation}
    \|u_\sigma\|_{L^{p^*_s}(\R^3,\CC)}=\|u\|_{L^{p^*_s}(\R^3,\CC)}=1 \quad \text{ and } \quad [u_\sigma]_{s,p}=[u]_{s,p},
\end{equation}
while
\begin{equation}
    [u_\sigma]_{A,s,p}^p = \int_{\R^3\times \R^3} \frac{|\e^{-i\sigma(x-y)\cdot A\left(\sigma\frac{x+y}{2}\right)}u(x)-u(y)|^p}{|x-y|^{3+sp}} \dd x\dd y.
\end{equation}
We compute
\begin{equation}
\label{seminorm:rescaled:convergence}
    \begin{aligned}
        [u_\sigma]_{A,s,p}^p - [u]_{s,p} 
        &=\int_{\R^3\times \R^3} \frac{|\e^{-i\sigma(x-y)\cdot A\left(\sigma\frac{x+y}{2}\right)}u(x)-u(y)|^p- |u(x)-u(y)|^p}{|x-y|^{3+sp}} \dd x\dd y\\
        &=\int_{\supp u\times \supp u}\Theta_\sigma(x,y) \dd x \dd y,
    \end{aligned}
\end{equation}
where
$$\Theta_\sigma(x,y):=\frac{|\e^{-i\sigma(x-y)\cdot A\left(\sigma\frac{x+y}{2}\right)}u(x)-u(y)|^p- |u(x)-u(y)|^p}{|x-y|^{3+sp}}.$$
It follows that $\Theta_\sigma(x,y) \to 0$ as $\sigma \to 0^+$ for a.e. $(x,y)\in \supp u\times\supp u$. Furthermore, arguing as in \cite[Proposition 2.2]{dAveniaSquassina}, there exists a constant $C>0$ such that 
    \begin{equation}
        |\expAminus u(x) - u(y)| \leq C\min\{1,|x-y|\},\qquad \text{ for every } (x, y)\in \supp u\times\supp u.
    \end{equation}
Hence,
\begin{equation}
\label{theta:est}
    |\Theta_\sigma(x,y)| \leq w(x,y), \quad \text{ where } w(x,y) = C\min\left\{\frac{1}{|x-y|^{3+ps-p}},\frac{1}{|x-y|^{3+ps}} \right\}.
\end{equation}
Therefore, by Dominated Convergence Theorem, \eqref{Sepsilon}, \eqref{seminorm:rescaled:convergence}
$$S_A\le \lim_{\sigma\to0}[u_\sigma]_{A,s,p}^p = [u]_{s,p}^p\le S + \eps.$$
Thus, $S_A\le S$ by the arbitrariness of $\eps$ and \eqref{SSA:normalized}.

To conclude, we observe that by \eqref{diamagnetic:D}, we get $S \leq S_A$. The proof is finished.
\end{proof}

\begin{remark}
For simplicity, from now on, we refer to $S$ or $S_A$ indiscriminately.
\end{remark}

\section{Quasilinear critical problems in \texorpdfstring{$\R^3$}{R3}}
\label{Sect:BN}

The present section is dedicated to the study of nontrivial weak solution in $D^{s,p}_A(\R^3,\C)$ to equation \eqref{main:equation} driven by the magnetic fractional $p$-Laplacian operator defined in \eqref{BaldelliBernini:def}, here reported for completeness.
\begin{equation}
    \label{p:BN:equation}
    (-\Delta)^s_{p,A}u =\lambda H(x)|u|^{q-2}u+K(x)|u|^{p^*_s-2}u\quad \text{ in } \R^3,
\end{equation}
where $0<s<1<p$, $ps<3$, $1<q<p^*_s$,  $p^*_s$ is the $p$-fractional Sobolev critical exponent, $A:\R^3 \to \R^3$ be a vector field with locally bounded gradient, $\lambda>0$, $H,K$ satisfy \eqref{H:hyp}, \eqref{K:hyp} and \eqref{H2:hyp}.

We would like to perform a variational approach. To this end, let us consider the energy functional $\cE:D^{s,p}_A(\R^3,\C) \to \R$ (we refer to Section \ref{hom:magn} for the definition of $D^{s,p}_A(\R^3,\C)$ and its properties) associated to \eqref{p:BN:equation} defined as
\begin{equation}
    \label{energy:functional:BN}
    \cE(u):=\frac1p[u]_{A,s,p}^p - \frac{\lambda}{q}\int_{\R^3}H(x)|u|^q \dd x - \frac{1}{p^*_{s}}\int_{\R^3}K(x) |u|^{p^*_{s}} \dd x.
\end{equation}
Of course, the functional $\cE$ is  well defined in the entire space $D^{s,p}_A(\R^3,\C)$. The proof of the $C^1$ regularity of $\cE$ in $D^{s,p}_A(\R^3,\C)$ is almost standard.
In turn, $\cE': D^{s,p}_A(\R^3,\C)\to (D^{s,p}_A(\R^3,\C))'$ is given by
$$  \cE'(u)[v]=\Re\left( \langle u, v\rangle_{A,s,p} - \lambda\int_{\R^3} H(x)|u|^{q-2}u\cdot\overline v \dd x 
  -\int_{\R^3}K(x) |u|^{p^*_{s}-2}u\cdot\bar v \dd x
  \right).$$

Then, weak solutions to \eqref{main:equation} are sought as critical points of $\cE$, i.e. those points $u \in D^{s,p}_A(\R^3,\C)$ such that $\cE'(u)[v]=0$
for any $v\in D^{s,p}_A(\R^3,\C)$, where
$$
    \langle u, v\rangle_{A,s,p} :=\Re\int_{\R^3\times \R^3}\frac{B^A_u(v)}{|x-y|^{3+ps}} \dd x \dd y,
$$
where, inspired by \cite{PucciXiangZhang}, $B^A_u(v)$ is defined as
$$
    B^A_u(v):=|\expAminus u(x) - u(y)|^{p-2}(\expAminus u(x) -  u(y))\overline{(\expAminus v(x) -  v(y))}.
$$

In particular, we are going to prove existence results for nontrivial solutions with positive and negative energy, i.e. Theorem \ref{main:theorem:q>p} and \ref{main:theorem:q<p}, whose statements are given in the Introduction. First, we state some preliminary results.

\subsection{Preliminaries}

Let us begin by recalling the well-known Mountain Pass Theorem.

\begin{theorem}[\cite{AmbrosettiRabinowitz}]\label{mpthm}
Let $(Y,\|\cdot\|_Y)$ be a Banach space and consider $E\in C^1(Y)$. We assume that
\begin{enumerate}[label=(\roman*)]
    \item $E(0)=0$,
    \item there exist $\alpha,R>0$ such that $E(u)\geq\alpha$ for all $u\in Y$, with $\|u\|_Y=R$,
    \item there exists $v_0\in Y$ such that $\limsup_{t\to \infty}E(tv_0)<0$.
\end{enumerate}
Let $t_0>0$ be such that $\|t_0v_0\|_Y>R$ and $E(t_0v_0)<0$
and let
$$c=\inf_{\gamma\in\Gamma}\,\sup_{t\in [0,1]}E(\gamma(t)),$$
where
$$\Gamma=\{\gamma\in C^0([0,1],Y)\,/\, \gamma(0)=0\hbox{ and }\gamma(1)=t_0v_0\}.$$
Then, there exists a Palais--Smale sequence at level $c$, that is
 a sequence $(u_n)_n\subset Y$ such that
$$\lim_{n\to \infty}E(u_n)=c\quad\hbox{ and }\quad\lim_{n\to \infty}E'(u_n)= 0\quad\hbox{strongly in }Y'.$$
\end{theorem}

A crucial tool for problems with a lack of compactness is a concentration-compactness principle.
Before reporting the concentration-compactness result, we report a slight improvement of \cite[Lemma 2.4]{BonderSaintierSilva}. 

\begin{lemma}
\label{BoSaSi:cpt:lemma}
Let $0 < s < 1 < p$ be such that $sp < 3$ and let $1< q < p^*_s$. Let $w \in L^\infty(\R^3,\R)$ be such that there exist $\alpha > 0$ and $C > 0$ such that
 $$0 \leq w(x) \leq C|x|^{-\alpha}.$$
If $\alpha > sq - \frac{3(q-p)}{p}$, then $D^{s,p}(\R^3,\R) \subset \subset L^q(w \dd x; \R^3,\R)$. That is, for any bounded sequence $(u_n)_k \subset D^{s,p}(\R^3,\R)$, there exists a subsequence $(u_{k_j})_j$ and a function $u \in D^{s,p}(\R^3,\R)$ such that $u_{k_j} \rightharpoonup u$ weakly in $D^{s,p}(\R^3,\R)$ and
\begin{equation}\label{strongconvpesata}
\int_{\R^3} |u_{k_j}(x) - u(x)|^q w(x) \dd x \to 0 \quad \text{as } j \to \infty.
\end{equation}
\end{lemma}

We want to remark that, despite the authors in \cite{BonderSaintierSilva} proved it for $p \leq q < p^*_s$, this is actually true in the entire range $1<q<p^*_s$. Indeed, the condition $sq - 3(q-p)/p>0$ holds if $q< p^*_s$, the lower bound $q \geq p$ is not decisive.

In what follows, we will state, and then proved in Appendix \ref{Appendix}, the quasilinear magnetic version of the concentration compactness principle. Before stating it, in order to mantain the shape of the classical one, we define the fractional magnetic $(s, p)$-gradient of a function $v \in D^{s, p}_A(\R^3,\C)$ as
\begin{equation}
\label{definition:DsA:seminorm}
   \left|D^s_A v(x)\right|^p=\int_{\R^3} \frac{|\expAminus v(x)-v(y)|^p}{|x-y|^{3+s p}} \dd y. 
\end{equation}

Observe that this magnetic $(s, p)$-gradient is well defined a.e. in $\R^3$ and $\left|D^s_A v\right| \in L^p(\R^3,\R)$.
\begin{remark}
\label{sp:gradient:convergence:remark}
    Note that $v_n\to v$ in $D^{s,p}_A(\R^3,\C)$ is equivalent to $D^s v_n\to D^s_A v$ in $L^p(\R^3,\C)$, the same holds for the weak convergence.
\end{remark}

Now we are ready to report the concentration compactness principle around points and at infinity that will be used below as the natural generalization of the ones in \cite{Lions1985, btw, BonderSaintierSilva}.

\begin{theorem}
\label{Our:concentration:compactness}
Let $0<s<1<p$, $ sp< 3$.
Let $(u_{n})_{n}$ be a weakly convergent sequence in $D^{s,p}_A(\R^3,\CC)$ to some $u$.
Then, there exist two bounded nonnegative measures $\mu,\nu$, an at most countable set
$J$, a family $(x_j)_{j\in J}$ of distinct points in
$\R^3$, nonnegative real numbers $\mu_i, \nu_i, i\in I$ such that

\begin{gather}
|u_{n}|^{p^*_s}\rightharpoonup\nu=|u|^{p^*_s}+ \sum_{j\in J}\nu_{j}\delta_{x_{j}},\label{u}\\
|D^s_A u_{n}|^{p}\rightharpoonup\mu\ge|D^s_A u|^{p}+\sum_{j\in J}\mu_{j}\delta_{x_{j}},\label{DsA}\\
S_A\nu_{j}^{p/p^*_s}\le \mu_{j},\label{relation:nu:mu:atomic}
\end{gather}

where $(x_j)_{j\in J}$ are distinct points in $\R^3$, $\delta_{x}$ is the Dirac-mass of mass 1 concentrated at
$x\in\R^3$ and $S_A$ is the Sobolev's constant defined in \eqref{SA}.

Moreover, if we define
\begin{equation}\label{def:mu:nu:inf}
\nu_\infty=\lim_{R\to\infty} \limsup_{n\to\infty} \int_{|x|>R} |u_n|^{p^*_s} \dd x,\qquad
\mu_\infty=\lim_{R\to\infty} \limsup_{n\to\infty} \int_{|x|>R} |D^s_A u_n|^{p} \dd x.
\end{equation}
Then, the quantities $\nu_\infty$ and $\mu_\infty$ exist and satisfy
\begin{equation}\label{un:crit:inf}
\limsup_{n\to\infty} \int_{\R^3} |u_n|^{p^*_s} \dd x=\int_{\R^3} d\nu +\nu_\infty,
\end{equation}
\begin{equation}\label{Dspa:inf}
\limsup_{n\to\infty} \int_{\R^3} |D^s_A u_n|^{p} \dd x=\int_{\R^3} d\mu +\mu_\infty,
\end{equation}
\begin{equation}\label{relation:nu:mu:infty}
S_A\nu_{\infty}^{p/p^*_s}\le \mu_\infty.
\end{equation}
\end{theorem}

\subsection{Functional geometry}

The present section is dedicated to study the geometry of the energy functional $\cE$ defined in \eqref{energy:functional:BN}, according to the range of the subcritical exponent $q$. Starting with the $p$-superlinear case, we will prove below that we are in the mountain pass setting.

\begin{lemma}\label{mpt_geometry}
Assume $0<s<1<p$, $sp<3$, $q \in [p, p^*_s)$, $A:\R^3 \to \R^3$ be a vector field with locally bounded gradient, \eqref{H:hyp}, \eqref{K:hyp}. Then, the functional $\cE$ has the Mountain Pass geometry,  namely the assumptions of Theorem \ref{mpthm} are satisfied:
\begin{enumerate}[label=\arabic*)]
    \item for every $\lambda> 0$ if $p<q<p^*_s$.
    \item for every $\lambda < \Lambda^* $ if $q=p$, where
    \begin{equation}\label{lambda*}
        \Lambda^*:=S_A/\|H\|_{L^r(\R^3,\R)}.
    \end{equation}
\end{enumerate}
\end{lemma}
\begin{proof}
We have to verify the hypotheses $(i)-(iii)$ of the Mountain Pass Theorem.
From Section \ref{Section:funct:sett}, $\cE\in C^1(D^{s,p}_A(\R^3,\C),\R)$ and clearly  $(i)$ is satisfied since $\cE(0)=0$.

We observe that assumption $(iii)$ of Theorem \ref{mpthm} can be proved independently from the range of $q$. Indeed, let $t>0$, $u\in D^{s,p}_A(\R^3,\C) \setminus\{0\}$ then, using \eqref{H:hyp} and \eqref{K:hyp}, we have
 
\begin{equation}\label{iii}
    \cE(tu)\le \frac{t^p}{p}[u]^p_{A,s,p} -\frac{\lambda t^q}{q}\int_{\R^3}H(x)|u|^p\dd x - \frac{t^{p^*_s}}{p^*_s}\int_{\R^3} K(x)|u|^{p^*_s} \dd x
     \to -\infty
\end{equation}
as $t$ goes to infinity,  since $p<p^*_s$ and by \eqref{homogeneous:Sobolev:embedding}.
Thus, choosing $t_u>0$ there exists $R>0$ such that $\cE(tu)<0$ for all $t\ge t_u$ and $\|t_u u\|>R$, then the proof of $(iii)$ is concluded.\\
Now, we take care of assumption $(ii)$, splitting the remaining part of the proof into two parts.\\
\textbf{\textit{1) Case ${\bs{ p<q<p^*_s}}$}.}\\
By \eqref{H:hyp}, \eqref{K:hyp}, Sobolev's and H\"older's inequalities, we get
\begin{equation}\label{mp:estim}\begin{aligned}
\cE(u)&= \frac1p[u]_{A,s,p}^p-\frac{\lambda}{q}\int_{\R^3}H(x)|u|^q\dd x-\frac{1}{p^*_s}\int_{\R^3}K(x)|u|^{p^*_s}\dd x
\\&\ge \frac1p[u]_{A,s,p}^p-\frac{\lambda}{q}\|H\|_{L^r(\R^3,\R)}\|u\|_{L^{p^*_s}(\R^3,\C)}^q\dd x-\frac{1}{p^*_s}\|K\|_{L^\infty(\R^3,\R)}\|u\|^{p^*_s}_{p^*_s}\\
&\ge [u]_{A,s,p}^p\left[\frac1p- \frac{\lambda}{q}\|H\|_{L^r(\R^3,\R)}S_A^{-q/p}[u]_{A,s,p}^{q-p}-\frac{1}{p^*_s}\|K\|_{L^\infty(\R^3,\R)} S_A^{-p^*_s/p}[u]_{A,s,p}^{p^*_s-p}\right].
\end{aligned}\end{equation}
Therefore, being $p<q<p^*_s$, there exist $\alpha, R>0$ with $R$ small enough,
 so that $\cE(u)\ge\alpha>0$
whenever $[u]_{A,s,p}=R$.\\
\textbf{\textit{2) Case ${\bs{ q=p}}$}.}\\
Using \eqref{mp:estim} with $q=p$ we obtain 
$$\cE(u)\ge [u]_{A,s,p}^p\left[\frac1p- \frac{\lambda}{p}\|H\|_{L^r(\R^3,\R)}S_A^{-1}\right]-\frac{1}{p^*_s}\|K\|_{L^\infty(\R^3,\R)} S_A^{-p^*_s/p}[u]_{A,s,p}^{p^*_s}.$$
Now define $f:[0,+\infty) \to \R$ as 
$$f(t):=t^p\left[\frac1p- \frac{\lambda}{p}\|H\|_{L^r(\R^3,\R)}S_A^{-1}\right]-\frac{1}{p^*_s}\|K\|_{L^\infty(\R^3,\R)} S_A^{-p^*_s/p}t^{p^*_s}.$$
Thus, for every $\lambda<\Lambda^*$, see \eqref{lambda*}, there exists $\bar{t}>0$, given by
\begin{equation}
\bar
t:=\left(\frac{1-\lambda S_A^{-1}\|H\|_{L^r(\R^3,\R)}}{S_A^{-p^*_s/p}\|K\|_{\infty(\R^3,\C)}}\right)^{\frac{1}{p^*_s-p}},
\end{equation}
such that $f'(t)>0$ for every $t < \bar{t}$. This fact, together with $f(0)=0$, gives the same conclusion as above.\\

Thus, condition $(ii)$ of Theorem \ref{mpthm} is satisfied for every $q\in[p,p^*_s)$.

Consider
$$\Gamma_u=\{\gamma\in C^0([0,1],D^{s,p}_A(\R^3,\C))\, : \, \gamma(0)=0\hbox{ and }\gamma(1)=t_u u\},$$
and
\begin{equation}\label{cu}
c_M=\inf_{\gamma\in\Gamma_u}\,\sup_{t\in [0,1]}\cE(\gamma(t)).
\end{equation}
Since the hypotheses of Theorem \ref{mpthm} are satisfied, there exists a $(PS)_{c_M}-$ sequence.\\

\end{proof}

In the $p$-sublinear case, the functional $\mathcal{E}$ may be negative everywhere. To introduce a truncated functional $\mathcal{E}_\infty$ and ensure that the Krasnosel’skii genus framework applies correctly, we must first choose $\lambda$ appropriately.

\begin{lemma}\label{genus_geometry}
Assume $0<s<1<p$, $sp <3$, $1<q<p$, $A:\R^3 \to \R^3$ be a vector field with locally bounded gradient, \eqref{H:hyp} and \eqref{K:hyp}. Define
\begin{equation}\label{lambdagenus}
        \lambda_{*,1}:=\frac{sq}{3\|H\|_{L^r(\R^3,\R)}}S_A^{\frac{3+q(s-3)}{sp}}\|K\|_{L^\infty(\R^3,\R)}^{\frac{(3-sp)(q-1)}{sp}}
    \end{equation}
Then, for every $\lambda<\lambda_{*,1}$, the functional $\cE$ satisfies the following
\begin{enumerate}[label=\arabic*)]
\item there exists $u^*\in \cD^{s,p}_A(\R^3,\CC)$ such that $\cE(u^*)>0$.
\item there exists a functional $\cE_\infty: \cD^{s,p}_A(\R^3,\CC)\to \R$ such that $\cE_\infty\ge \cE$ and $\cE_\infty(u)= \cE(u)$ for $[u]_{A,s,p}$ sufficiently small.
\end{enumerate}
\end{lemma}

\begin{proof}
Concerning $1)$, we argue as in \cite{BaldelliGuarnotta}.
Recalling \eqref{mp:estim}, we define $g:[0,+\infty) \to \R$ as 
\begin{equation}\label{def_g}
g(t):=\frac1p t^p- \frac{\lambda}{q}\|H\|_{L^r(\R^3,\R)}S_A^{-q/p}t^q-\frac{1}{p^*_s}\|K\|_{L^\infty(\R^3,\R)} S_A^{-p^*_s/p}t^{p^*_s},
\end{equation}
such that $\cE(u)\ge g([u]_{A,s,p})$.
We further consider the function $g_1(t)=\frac1p t^p-\frac{1}{p^*_s}\|K\|_{L^\infty(\R^3,\R)} S_A^{-p^*_s/p}t^{p^*_s}$. We briefly explain the idea: we look for the (unique) maximum of $g_1$ in order to evaluate $g_1$ on it. Then, by plugging it into $g$, we are able to find a condition on $\lambda$ such that we achieve $1)$. So, again simple computations give that the maximum of $g_1$ is achieved in
\begin{equation}
    t^*:=\left(\frac{1}{S_A^{-p^*_s/p}\|K\|_{L^\infty(\R^3,\R)}}\right)^{\frac{1}{p^*_s-p}}\qquad \text{and}\qquad
    g_1(t^*) = \frac{s}{3}\frac{S_A^{3/(sp)}}{\|K\|_{L^\infty(\R^3,\R)}^{3/(sp^*_s)}}.
\end{equation}
Hence, for every $\lambda < \lambda_{*,1}$
\begin{equation}\label{gtstar}
    g(t^*) = g_1(t^*)  - \frac{\lambda}{q}\|H\|_{L^r(\R^3,\R)}S_A^{-q/p}(t^*)^q > 0.
\end{equation}

In order to prove $2)$, from \eqref{def_g}, $g(t)\to 0^-$ as $t\to0^+$, $g(t)\to -\infty$ as $t\to\infty$, we observe that from \eqref{gtstar} there exist $T_0,T_1>0$ such that $g(T_0)=g(T_1)=0$ and
\begin{equation}
g(t)>0 \quad\text{if}\quad t\in (T_0,T_1) \qquad\qquad g(t)\le 0 \quad\text{if}\quad  t\in[0, T_0]\cup [T_1,\infty).
\end{equation}

Now take a cut off function $1\ge \tau\in C^\infty(\R^+_0,\R^+_0)$ such that
$\tau(t)=1$ if $0\le t\le T_0$ and $\tau(t)=0 $ if $t\ge T_1.$
Define the so-called truncated functional $\cE_\infty:\cD^{s,p}_A(\R^3,\CC)\to\R$ and truncated function $g_\infty:\R^+_0\to\R$ as follows
\begin{equation}\label{Einfty}
\cE_\infty(u)=\frac1p[u]_{A,s,p}^p-\frac{\lambda}{q}\int_{\R^3}H(x)|u|^q\dd x-\frac{ \tau([u]_{A,s,p})}{p^*_s}\int_{\R^3}K(x)|u|^{p^*_s}\dd x\ge g_\infty([u]_{A,s,p})
\end{equation}
by using Sobolev's and H\"older's inequalities with
$$g_\infty(t):=\frac1p t^p- \frac{\lambda}{q}\|H\|_{L^r(\R^3,\R)}S_A^{-q/p}t^q-\frac{1}{p^*_s}\|K\|_{L^\infty(\R^3,\R)} S_A^{-p^*_s/p}\tau(t)t^{p^*_s},$$
satisfying $g_\infty(t)\ge  g(t)$  and
$$g(t)= g_\infty(t) \quad\text{if}\quad t\in(0,T_0],\qquad g_\infty(t)\ge  g(t)>0 \quad\text{if}\quad t\in(T_0,T_1),$$
$$g_\infty(T_1)>0=g(T_1), \qquad g_\infty(t)>0 \quad\text{if}\quad t\in(T_1,\infty),\qquad g_\infty(t)\to\infty \quad\text{as}\quad t\to \infty.$$
As consequences, then $\cE_\infty\ge \cE$ and $\cE_\infty(u)= \cE(u)$ for any $[u]_{A,s,p}\le T_0$, which is $2)$.
\end{proof}

\subsection{\texorpdfstring{$(PS)-$}{(PS)-}sequences: boundedness and compactness}

The goal of the second part of the paper is to obtain weak solutions to \eqref{p:BN:equation} via variational methods, accounting for the lack of compactness arising from both the $p$-fractional critical Sobolev exponent and the unboundedness of the domain. Compactness is recovered by analyzing Palais–Smale sequences for the functional $\cE$. Specifically, we show that these sequences are bounded and, under suitable assumptions on $\lambda$, $c$, and $q$, convergent. We begin by proving their boundedness

\begin{lemma}\label{lembound} 
Assume $0<s<1<p$, $sp<3$, $1<q<p^*_s$, $A:\R^3 \to \R^3$ be a vector field with locally bounded gradient, \eqref{H:hyp} and \eqref{K:hyp}.
Let $(u_n)_n\subset D^{s,p}_A(\R^3,\C)$ be a Palais--Smale sequence for $\cE$, then 
$(u_n)_n$ is bounded in $D^{s,p}_A(\R^3,\C)$ for any $c\in\R$ and every $\lambda>0$.

Moreover, if $1<q<p$ and $c<0$, then we can estimate the $D^{s,p}_A(\R^3,\C)$ norm of $(u_n)_n$ as follows
\begin{equation}\label{estnorm}
[u_n]_{A,s,p}\le \lambda^{\frac{1}{p-q}}\left[3 S_A^{-q/p}\|H\|_{L^r(\R^3,\R)}\left(\frac{1}{q}-\frac{1}{p^*_s}\right)\right]^{\frac{1}{p-q}}.
\end{equation}
\end{lemma}

\begin{proof}
Let $(u_{n})_{n}\subset D^{s,p}_A(\R^3,\C)$ be a $(PS)_{c}$ sequence of $\cE$ for all $c\in\R$ namely, 
\begin{equation}\label{PS_prop}\cE(u_{n})=c+o(1), \,\ \cE'(u_{n})=o(1)\quad \text{as} \quad n\to\infty,\end{equation}
so that $|\langle \cE'(u_{n}),u_{n}\rangle|\le [u_n]_{A,s,p}$ for $n$ large. We divide the proof in two cases.\\
\textbf{\textit{1) Case $p\le q<p^*_s$}.}\\
By \eqref{K:hyp}, \eqref{magnetic:Sobovel:inequality:Dspa} we have
$$\begin{aligned}
c +o(1) +o(1)[u_n]_{A,s,p}&= \cE(u_{n})- \frac{1}{q}\langle \cE'(u_{n}),u_{n}\rangle\\
&= \left(\frac{1}{p}-\frac{1}{q}\right)[u_n]_{A,s,p}^p
-\biggl(\frac1{p^*_s}-\frac 1{q}  \biggr)\int_{\R^3}K(x)|u_n|^{p^*_s}\dd x\\
&\ge \left(\frac{1}{p}-\frac{1}{q}\right)[u_n]_{A,s,p}^p- S_A^{-3/(3-sp)}\biggl(\frac 1{p^*_s} -\frac1{q} \biggr)\|K\|_{L^{\infty}(\R^3,\C)}[u_n]_{A,s,p}^{p^*_s}.
\end{aligned}$$
Thus, it immediately follows that
$[u_n]_{A,s,p}$ should be bounded.

\textbf{\textit{1) Case $1<q<p$}.}
By \eqref{H:hyp}, \eqref{magnetic:Sobovel:inequality:Dspa}
\begin{equation}\label{bound_1qp}\begin{aligned}
c +o(1) +o(1)[u_n]_{A,s,p}&= \cE(u_{n})- \frac{1}{p^*_s}\langle \cE'(u_{n}),u_{n}\rangle\\
&= \left(\frac{1}{p}-\frac{1}{p^*_s}\right)[u_n]_{A,s,p}^p
-\lambda\biggl(\frac1q - \frac{1}{p^*_s}  \biggr)\int_{\R^3}H(x)|u_n|^{q}\dd x\\
&\ge \left(\frac{1}{p}-\frac{1}{p^*_s}\right)[u_n]_{A,s,p}^p-\lambda S_A^{-q/p}\biggl(\frac1q - \frac{1}{p^*_s} \biggr)\|H\|_{L^r(\R^3,\R)}[u_n]_{A,s,p}^q.
\end{aligned}\end{equation}
Since $1<q<p$, then again $[u_n]_{A,s,p}$ should be bounded.\\

Finally, in order to prove \eqref{estnorm}, it is enough to start with \eqref{bound_1qp} and use $c<0$, recalling that $q<p$.
\end{proof}

We are now ready to prove that the energy functional $\cE$ satisfies a compactness property. The open range of $q$ extends from $1$ to $p^*_s$, but the conditions on the parameters may vary depending on whether $q$ is $p$-superlinear or $p$-sublinear.

\begin{lemma}\label{lem5}
Assume $0<s<1<p$, $p<3s$, $1<q < p^*_s$, $A:\R^3 \to \R^3$ be a vector field with locally bounded gradient, \eqref{H:hyp} and \eqref{K:hyp}.
Define 
\begin{gather}
c_{PS}:=\frac{s}{3}\frac{S_A^{3/sp}}{\|K\|_{L^\infty(\R^3, \CC)}^{3/(sp^*_s)}} \label{csegnato}\\
\lambda_{*,2}:=\frac{c_{PS}^{\frac{p-q}{p}}}{3^{q/p}\left(\frac{1}{q}-\frac{1}{p^*_s}\right)S_A^{-q/p}\|H\|_{L^r(\R^3,\R)}}\label{lambdacomp}.
\end{gather}
Then, the energy functional $\cE$ satisfies the $(PS)_c-$condition 
\begin{enumerate}[label=\arabic*)]
    \item with $c<c_{PS}$ for every $\lambda>0$ if $p\le q<p^*_s$
    \item with $c<0$ for every $\lambda<\lambda_{*,2}$, if $1 < q < p$.
\end{enumerate}
\end{lemma}

\begin{proof}
Let $(u_n)_n$ be a $(PS)_c$ sequence in $D^{s,p}_A(\R^3,\C)$, that is \eqref{PS_prop} holds. By Lemma \ref{lembound}, then $(u_n)_n$ is bounded in $D^{s,p}_A(\R^3,\C)$, which is a reflexive Banach space, see Section \ref{Section:funct:sett}.
By Banach-Alaoglu's Theorem,
 there exists $u\in D^{s,p}_A(\R^3,\C)$ 
 such that, up to subsequences, we get
$u_n\rightharpoonup u$ in $D^{s,p}_A(\R^3,\C)$, that is $D^s_Au_{n}\rightharpoonup D^s_A u$ in $L^p(\R^3,\C)$ by Remark \ref{sp:gradient:convergence:remark}.
By Theorem \ref{Our:concentration:compactness}, there exist two bounded nonnegative measures $\mu,\nu$, an at most countable set $J$, a family $(x_j)_{j\in J}$ of distinct points in
$\R^3$ and $(\nu_j)_{j\in J}, \,(\mu_j)_{j\in J}\in [0,\infty)$ such that \eqref{u}, \eqref{DsA}, 
\eqref{un:crit:inf}, \eqref{Dspa:inf} hold, with $\nu_\infty$, $\mu_\infty$ defined as in \eqref{def:mu:nu:inf}
satisfying
\begin{equation}\label{6.22}
S_A\nu_{j}^{p/p^*_s}\le\mu_{j}, \qquad S_A\nu_{\infty}^{p/p^*_s}\le \mu_\infty.
\end{equation}

Take a standard cut-off function $\psi\in C_c^{\infty}(\R^3,\R)$, such that $0\le\psi\le1$ in $\R^3$,
$\psi=0$ for $|x|>1$, $\psi=1$ for $|x|\le 1/2$. For any $j\in J$ and each $0<\eps<1$, define
$$\psi_{\eps}(x)=\psi\left(\frac{x-x_{j}}{\eps}\right).$$

Since $\cE'(u_{n})\phi\to0$ being $(u_{n})_n$ a $(PS)_c-$sequence and 
choosing $\phi=\psi_\eps u_n$, which is still bounded,  we have, as $n\to\infty$,
\begin{equation}\label{dis1}
    \Re\int_{\R^3\times \R^3}\frac{B^A_u(u_n\psi_\eps)}{|x-y|^{3+ps}} \dd x \dd y=\lambda\int_{\R^3} H(x)|u_n|^q\psi_\eps  \dd x+\int_{\R^3}K(x)|u_n|^{p^*_s}\psi_\eps \dd x+o(1).
\end{equation}
Arguing as in \cite{BonderSaintierSilva}, we write
\begin{equation}
    \begin{aligned}
    &\int_{\R^3\times \R^3}\frac{B^A_u(u_n\psi_\eps)}{|x-y|^{3+ps}} \dd x \dd y= \cI_1 + \cI_2:=\int_{\R^3\times \R^3}\frac{|\expAminus u_n(x) - u_n(y)|^p\psi_\eps(x)}{|x-y|^{3+ps}} \dd x \dd y \\
        & \quad +
        \int_{\R^3\times \R^3}\frac{|\expAminus u_n(x) - u_n(y)|^{p-2}(\expAminus u(x) -  u_n(y))(\psi_\eps(x)-\psi_\eps(y))\overline{u_n(y)}}{|x-y|^{3+ps}} \dd x \dd y.
    \end{aligned}
\end{equation}

Note that, by Lemma \ref{BoSaSi:cpt:lemma} and Lemma 2.2, (2.4) in \cite{BonderSaintierSilva}, 
\begin{equation}
\label{I2:estimates}
    |\cI_2|\le \int_{\R^3\times \R^3}\frac{|\expAminus u(x) - u(y)|^{p-1}|\psi_\eps(x)-\psi_\eps(y)| |\overline{u_n(y)}|}{|x-y|^{3+ps}} \dd x \dd y \to 0 \text{ as } n \to +\infty.
\end{equation}

Now, by \eqref{H:hyp} by applying the Dominated Convergence Theorem, we have
\begin{equation}\label{pl}
\lim_{\eps\to 0}\lim_{n\to\infty}\int_{\R^3}H(x)|u_n|^q\psi_\eps  \dd x=\lim_{\eps\to 0}\int_{B_\eps(x_j)}H(x)|u|^q\psi_\eps  \dd x=0.
\end{equation}

Then, by \eqref{I2:estimates} and \eqref{pl}, we can conclude for $n$ large
\begin{equation}\label{nabla_le_K_ustar}
    \int_{\R^3}|D^s_A u_n|^p\psi_\eps(x) \dd x
    \le \int_{\R^3}K(x)|u_n|^{p^*_s}\psi_\eps \dd x+o(1).
\end{equation}
Hence, by \eqref{K:hyp} and $\mbox{supp} \psi_\eps=B_\eps(x_j)$, if $\eps\to0$ we deduce
\begin{equation}\label{6.23}
\mu_{j}\le  K(x_{j})\nu_{j}.
\end{equation}
Consequently, either  $\nu_j=0$ and then also $\mu_j=0$,  or $\nu_j>0$. We will show that  
the latter case cannot occur for any $j\in J$, with $J$ given in Theorem \ref{Our:concentration:compactness}. Note that combining \eqref{6.23} and \eqref{6.22}, we have
\begin{equation}\label{sxj}
S_A\leq  K(x_j)\nu_j^{sp/3},
\end{equation}
which establishes that the concentration of the measure $\nu$ can occur only at points $x_j$ where $K(x_j)>0$.
Consequently, from \eqref{6.22} and \eqref{6.23} the measure $\mu$ can concentrate at points
in which the measure $\nu$ can.
Hence, the set $X_J=\{x_j: j\in J\}$  does not contain those points $x_{j}$ which are zeros for $K$. Let $J_1=\{j\in J: K(x_j)>0\}$, we claim that $J_1=\emptyset$. 

We proceed by contradiction. Take any $j\in J_1$, then \eqref{sxj} implies
\begin{equation}\label{nujcontradiction}
 \nu_{j}\ge \left( \frac{S_A}{ K(x_{j})}\right)^{3/sp}\ge \left( \frac{S_A}{\|K\|_{L^\infty(\R^3,\R)}}\right)^{3/sp}.
\end{equation}
Now, we show that \eqref{nujcontradiction} cannot occur. 

First, note that $|J_1|<\infty$; indeed, being $\nu$ a bounded measure, \eqref{un:crit:inf} and \eqref{nujcontradiction},
 we get 
$$
    \infty>\int_{\R^3}d\nu=\|u\|_{L^{p^*_s}(\R^3,\C)}^{p^*_s}+\int_{\R^3}\sum_{j\in J_1}\nu_j\delta_{x_j}\dd x+\nu_\infty
    \ge \|u\|_{L^{p^*_s}(\R^3,\C)}^{p^*_s}+\left(\frac{S_A}{\|K\|_{L^\infty(\R^3,\R)}}\right)^{3/sp}|J_1|+\nu_\infty.
$$

Now we have to divide the proof into two cases.

{\bf Case ${\bf p\le q<p^*_s}$.}

Using $\psi_\eps\le 1$ and \eqref{H:hyp}, being $q\ge p$, we get
\begin{equation}\begin{aligned}
c+o(1)&= \cE(u_n)-\frac{1}{p}\langle \cE'(u_n),u_n\rangle\\
&=-\left(\frac{1}{q}-\frac{1}{p}\right)\int_{\R^3}H(x)|u_n|^{q}\dd x
-\biggl(\frac1{p^*_s}-\frac 1{p}  \biggr)\int_{\R^3}K(x)|u_n|^{p^*_s}\dd x\ge \frac{ s}{3}\int_{\R^3}K(x)|u_n|^{p^*_s}\psi_\eps\dd x
\end{aligned}\end{equation}

Using \eqref{nujcontradiction} and letting $n\to\infty$ and $\eps\to0$, we obtain
\begin{equation}\label{epn:q>p}
    c\ge \frac{ s}{3} \nu_j K(x_j)\ge \frac{ s}{3}\left( \frac{S_A}{ K(x_{j})}\right)^{3/sp} K(x_j)=\frac{s}{3}\frac{S_A^{3/sp}}{( K(x_{j}))^{3/(sp^*_s)}}\ge \frac{s}{3}\frac{S_A^{3/sp}}{( \|K\|_{L^\infty(\R^3, \CC)})^{3/(sp^*_s)}} (=c_{PS})
\end{equation}
which contradicts the assumption $c<c_{PS}$, thus $J_1 = \emptyset$. 

{\bf Case ${\bf 1<q<p}$.}
Note that, by \eqref{DsA}, \eqref{nujcontradiction}, it holds
$$
    \mu_j
    \ge 
    S_A\nu_j^{p/p^*_s}
    \ge 
    S_A\left( \frac{S_A}{\|K\|_{L^\infty(\R^3,\R)}}\right)^{3/sp^*_s}
    =
    \frac{S_A^{3/sp}}{\|K\|_{L^\infty(\R^3,\R)}^{3/sp^*_s}}=\frac{3}{s}c_{PS}
$$
which, together with the definition of the magnetic $(s,p)-$gradient, the Dominated Convergence Theorem, \eqref{H:hyp}, H\"older inequality, \eqref{estnorm}, and \eqref{magnetic:Sobovel:inequality:Dspa}, entails
\begin{equation}\label{epn:q=p}
\begin{aligned}
    0&>c+o(1)= \cE(u_n)-\frac{1}{p^*_s}\langle \cE'(u_n),u_n\rangle\\&=\biggl(\frac1{p}-\frac 1{p^*_s}  \biggr)\int_{\R^3}|D^s_A u_n|^p \dd x-\lambda\left(\frac{1}{q}-\frac{1}{p^*_s}\right)\int_{\R^3}H(x)|u_n|^{q}\dd x\\
    &\ge\frac s3  \int_{\R^3}|D^s_A u_n|^p\psi_\eps \dd x-\lambda\left(\frac1q-\frac{1}{p^*_s}\right)\int_{\R^3}H(x)|u|^{q}\dd x\\
    &\ge\frac s3 \mu_j+ \frac s3[u]_{A,s,p}^p-\lambda\left(\frac{1}{q}-\frac{1}{p^*_s}\right)S_A^{-q/p}\|H\|_{L^r(\R^3,\R)}[u]_{A,s,p}^{q}\\
    &\ge c_{PS}+ \frac s6[u]_{A,s,p}^p-\lambda\left(\frac{1}{q}-\frac{1}{p^*_s}\right)S_A^{-q/p}\|H\|_{L^r(\R^3,\R)}\lambda^{\frac{q}{p-q}}\left[3 S_A^{-q/p}\|H\|_{L^r(\R^3,\R)}\left(\frac{1}{q}-\frac{1}{p^*_s}\right)\right]^{\frac{q}{p-q}}\\
    &=c_{PS}+ \frac s6[u]_{A,s,p}^p-\lambda^{\frac{p}{p-q}}3^{\frac{q}{p-q}}\left[ S_A^{-q/p}\|H\|_{L^r(\R^3,\R)}\left(\frac{1}{q}-\frac{1}{p^*_s}\right)\right]^{\frac{p}{p-q}},
\end{aligned}
\end{equation}
which contradicts assumption $\lambda<\lambda_{*,2}$ since $c<0$, thus $J_1 = \emptyset$ also in this latter case. \\

It remains to show that the concentration
of $\nu$ cannot occur at infinity, namely $\nu_\infty=0$. 

Following the same idea used to prove \eqref{6.23} but with a cut-off function such that $\psi_{R}\in C^{\infty}(\R^3,\R)$, for fixed $R\in\R_0^+$,
such that $0\le\psi_{R}\le1$ in $\R^{3}$, $\psi_R(x)=0$ for $|x|<R$ and $\psi_{R}(x)=1$ for $|x|>2R$, so that \eqref{nabla_le_K_ustar} holds with $\psi_\eps$ replaced by $\psi_R$.
Noting that
$$
    \lim_{R\to\infty}\limsup_{n\to\infty}\left\{\int_{\R^3}K(x)|u_{n}|^{p^*_s}\psi_{R}\dd x\right\} \le\|K\|_{L^\infty(\R^3, \CC)}\nu_{\infty},
$$
we finally get
$\|K\|_{L^\infty(\R^3, \CC)}\nu_\infty\ge\mu_\infty$,
which, together with \eqref{6.22} gives $\nu_{\infty}\ge  S_A^{3/sp}\|K\|_{L^\infty(\R^3, \CC)}^{-3/sp^*_s}$.
As for \eqref{epn:q>p} and \eqref{epn:q=p}, we get again a contradiction. 
Consequently,
$$
    \lim_{n\to\infty}\int_{\R^3}|u_{n}|^{p^*_s}\dd x=\int_{\R^3}|u|^{p^*_s}\dd x,
$$
that is $\|u_n\|_{L^{p^*_s}(\R^3, \CC)} \to \|u\|_{L^{p^*_s}(\R^3, \CC)}$ as $n\to\infty$, which implies $\|u_n-u\|_{L^{p^*_s}(\R^3, \CC)}\to 0$ by \cite[Proposition 3.32]{Brezis}. It remains to prove 
\begin{equation}\label{claimfin}
[u_n-u]_{A,s,p}\to0, \quad \text{as}\,\,\,n\to\infty.
\end{equation}
Since $(u_n)_n$ is a $(PS)_c-$sequence, we have
\begin{equation}\label{g-11}
\begin{aligned}
    o(1)
    &=\langle \cE'(u_n)-\cE'(u),u_n-u\rangle\\
    &=\Re \int_{\R^3\times \R^3}\frac{B^A_{u_n}(u_n-u)-B^A_{u}(u_n-u)}{|x-y|^{3+sp}}\dd x \dd y
    -\Re\int_{\R^3}H(x) (|u_n|^{q-2}u_n-|u|^{q-2}u)\overline{(u_n-u)}\dd x
    \\ 
    &\quad-\Re\int_{\R^3}K(x) (|u_n|^{p^*_s-2}u_n-|u|^{p^*_s-2}u)\overline{(u_n-u)}\dd x.\\
\end{aligned}
\end{equation}
By \eqref{K:hyp}, Lemma \ref{lembound}, using H\"older's and Schwarz's inequality, and the convergence of $(u_n)_n$ in $L^{p^*_s}(\R^3,\C)$ we get
 
$$
\begin{aligned}
    &\lefteqn{\left|\int_{\R^3}K(x)(|u_n|^{p^*_s-2}u_n-|u|^{p^*_s-2}u)\overline{(u_n-u)}\dd x\right|}\\
    &\le \|K\|_{L^\infty(\R^3,\R)} \int_{\R^3} \left(|u_n|^{p^*_s-1}+|u|^{p^*_s-1}\right)|u_n-u|\dd x\\
    & \le \|K\|_{L^\infty(\R^3,\R)} \left(\|u_n\|_{L^{p^*_s}(\R^3,\C)}^{p^*_s-1}+\|u\|_{L^{p^*_s}(\R^3,\C)}^{p^*_s-1}\right)\|u_n-u\|_{L^{p^*_s}(\R^3,\C)}=o(1),
\end{aligned}
$$
as $n\to\infty$. Similarly, by using \eqref{H:hyp}, we have
$$
    \begin{aligned}
    &\lefteqn{\left|\int_{\R^3}H(x)(|u_n|^{q-2}u_n-|u|^{q-2}u)\overline{(u_n-u)}\dd x\right|}\\
    &\quad \le\int_{\R^3} \left[H(x)(|u_n|^{q-1}+|u|^{q-1})\right]|u_n-u|\dd x\\
    & \quad \le\|H\|_{r}\left[\|u_n\|_{L^{p^*_s}(\R^3,\C)}^{q-1}+\|u\|_{L^{p^*_s}(\R^3,\C)}^{q-1}\right ]\|u_n-u\|_{L^{p^*_s}(\R^3,\C)}=o(1),
    \end{aligned}
$$
as $n\to\infty$. Thus, \eqref{g-11} reduces to 
\begin{equation}\label{crucial}
\begin{aligned}
    o(1)&=\langle \cE'(u_n)-\cE'(u),u_n-u\rangle=\Re \int_{\R^3\times \R^3}\frac{B^A_{u_n}(u_n-u)-B^A_{u}(u_n-u)}{|x-y|^{3+sp}}\dd x \dd y
\end{aligned}
\end{equation}

Now we make use of the following complex version of Simon's inequality, see \cite[Appendix]{bfk}
\begin{equation}\label{diaz_complex}
    |a\!-\!b|^{p}\!\lesssim\!\!\begin{cases}\Re\langle |a|^{p-2}a- |b|^{p-2}b
    ,a-b\rangle_{\mathbb C^3}&\phantom{1<\,}p\!\geq \!2;\\
    \!\left(\!\Re\langle |a|^{p-2}a- |b|^{p-2}b
    ,a-b\rangle_{\mathbb C^3}\!\right)\!^{\frac{p}{2}}\!\left(|a|^p+|b|^p\right)^{\frac{2-p}{2}} 
    &1\!<\!p\!<\!2.\end{cases}\qquad a,b\in\mathbb C^3
\end{equation}

If $p\ge 2$, by \eqref{diaz_complex} applied with $a=\expAminus u_n(x) - u_n(y)$ and $b=\expAminus u(x) - u(y),$
we immediately have
$$
\begin{aligned}
    [u_n-u]_{A,s,p}^p& =\int_{\R^3\times \R^3}\frac{|\expAminus (u_n(x)-u(x)) - (u_n(y)-u(y))|^p}{|x-y|^{3+ps}} \dd x \dd y\\
    &\lesssim  \Re \int_{\R^3\times \R^3}\frac{B^A_{u_n}(u_n-u)-B^A_{u}(u_n-u)}{|x-y|^{3+sp}}\dd x \dd y
\end{aligned}
$$
so that \eqref{claimfin} follows by \eqref{crucial}.

While, if $1<p<2$, then, by \eqref{diaz_complex} 
applied as before and by Holder's inequality with exponents $2/p$
and $2/(2-p)$,  we arrive to 
{\small$$\begin{aligned}
    [u_n-u]_{A,s,p}^p           &=\int_{\R^3\times \R^3}\frac{|\expAminus(u_n(x)-u(x))-(u_n(y)-u(y))|^p}{|x-y|^{3+ps}}\dd x\dd y\\
    &\lesssim \Re \int_{\R^3\times \R^3}\frac{(B^A_{u_n}(u_n-u)-B^A_{u}(u_n-u))^{p/2}[B^A_{u_n}(u_n)+B^A_{u}(u)]^{\frac{2-p}{2}}}{|x-y|^{3+sp}}\dd x \dd y\\
    &= \Re \int_{\R^3\times \R^3}\frac{(B^A_{u_n}(u_n-u)-B^A_{u}(u_n-u))^{p/2}[B^A_{u_n}(u_n)+B^A_{u}(u)]^{\frac{2-p}{2}}}{|x-y|^{(3+sp)p/2}\cdot |x-y|^{(3+sp)(2-p)/2}}\dd x \dd y\\
    &\le \Re\biggl[\biggl( \int_{\R^3\times \R^3}\frac{B^A_{u_n}(u_n-u)-B^A_{u}(u_n-u)}{|x-y|^{3+sp}}\dd x \dd y\biggl)^{p/2}\biggl(\int_{\R^3\times \R^3}\frac{B^A_{u_n}(u_n)+B^A_{u}(u)}{|x-y|^{3+sp}}\dd x \dd y\biggr)^{\frac{2-p}{2}}\biggr]\\
    &= \Re\biggl[\biggl( \int_{\R^3\times \R^3}\frac{B^A_{u_n}(u_n-u)-B^A_{u}(u_n-u)}{|x-y|^{3+sp}}\dd x \dd y\biggl)^{p/2}\biggl([u_n]_{A,s,p}^p+[u]_{A,s,p}^p\biggr)^{\frac{2-p}{2}}\biggr]
    \\&  \lesssim\Re\biggl[\biggl( \int_{\R^3\times \R^3}\frac{B^A_{u_n}(u_n-u)-B^A_{u}(u_n-u)}{|x-y|^{3+sp}}\dd x \dd y\biggl)^{p/2} \left([u_n]_{A,s,p}^{p(2-p)/2}+[u]_{A,s,p}^{p(2-p)/2}\right)\biggr],
\end{aligned}$$}
where we have used \eqref{standard}.
By the boundedness of $(u_n)_n$ by Lemma \ref{lembound}, we get 
$$
    \begin{aligned} 
        [u_n-u]_{A,s,p}^p \lesssim \Re\biggl( \int_{\R^3\times \R^3}\frac{B^A_{u_n}(u_n-u)-B^A_{u}(u_n-u)}{|x-y|^{3+sp}}\dd x \dd y\biggl)^{p/2}
    \end{aligned}
$$ 
where $C$ does not depend on $n$.
In turn, \eqref{crucial} immediately gives \eqref{claimfin} when $1<p<2$.
Thus, the proof of the lemma is concluded.
\end{proof}

\begin{remark}
We would like to remark that, as in \cite{bfk} with $s=1$, Lemma \ref{lem5} can be proved by the non-magnetic concentration compactness principle in $D^{s,p}(\R^3,\R)$ given by \cite[Theorem 1.1]{BonderSaintierSilva} and with the use of the Diamagnetic inequality \eqref{diamagnetic:D}. However, magnetic concentration compactness principle in $D^{s,p}_A(\R^3,\R)$, Theorem \ref{Our:concentration:compactness}, represents the most natural tool in our setting.
\end{remark}

Now we will prove that, in the case $1<q<p$, also the truncated functional $\cE_\infty$ defined in Lemma \ref{genus_geometry} satisfies the Palais-Smale condition under the same condition of $\cE$.

\begin{lemma}\label{compeinf}
Assume $0<s<1<p$, $p<3s$, $1<q < p$, $A:\R^3 \to \R^3$ be a vector field with locally bounded gradient, \eqref{H:hyp} and \eqref{K:hyp}. Then, the following hold
\begin{enumerate}[label=\arabic*)]
    \item if $\cE_\infty(u)<0$, then $[u]_{A,s,p}<T_0$ and $\cE_\infty(v)=\cE(v)$ for any $v$ in a neightborhood of $u$.
    \item the energy functional $\cE_\infty$ satisfies the $(PS)_c-$condition  with $c<0$ for every $\lambda<\lambda_{*,2}$, where $\lambda_{*,2}$ is defined in \eqref{lambdacomp}.
\end{enumerate}
\end{lemma}

\begin{proof}
To prove $1)$, assume by contradiction that $[u]_{A,s,p}\ge T_0$. By \eqref{Einfty}, then $\cE_\infty([u]_{A,s,p})\ge g_\infty([u]_{A,s,p})\ge 0$, which is a contradiction. The last part of $1)$ follows from the continuity of $\cE_\infty$ and $g(t)= g_\infty(t)$ if $t\in(0,T_0]$.

Let us focus on $2)$. Take $(u_n)_n$ be a $(PS)-$sequence at level $c<0$ for $\cE_\infty$. Then, for $n$ sufficiently large $\cE_\infty(u_n)<0$. Thus, applying $1)$, then $[u_n]_{A,s,p}<T_0$ for $n$ sufficiently large and $\cE_\infty(u_n)=\cE(u_n)$, implying that $(u_n)_n$ is a $(PS)-$sequence also for $\cE$. Thus $(u_n)_n$ converges at $\cD^{s,p}_A(\R^3,\CC)$ for any $\lambda<\lambda_{*,2}$ by Lemma \ref{lem5}, concluding the proof.
\end{proof}

\subsection{Proof of Theorem \ref{main:theorem:q>p} }

The last step to conclude Theorem \ref{main:theorem:q>p}, whose statement is given in the Introduction, consists in proving that the Mountain Pass level $c_M$ is below the level $c_{PS}$, under which the Palais-Smale condition holds. To do so, we denote another useful level
\begin{equation}\label{clambda}
c_A = \inf_{u\in D^{s,p}_A(\R^3,\C)\setminus\{0\}} \max_{t\ge 0} \cE(tu).
\end{equation}
\begin{remark}\label{culambda}
Obviously, $c_A\geq c_M$, where $c_M$ is defined in \eqref{cu},
since $\cE(tu)<0$ for $u\in D^{s,p}_A(\R^3,\C)\setminus\{0\}$ and $t$ large by the structure of $\cE$.
\end{remark}

\begin{lemma}\label{c<csegnato}
Assume $0<s<1<p$, $p<3s$, $p< q<p^*_s$, $A:\R^3\to \R^3$ be a vector field with locally bounded gradient, \eqref{H:hyp}, \eqref{H2:hyp} and \eqref{K:hyp}. Let $c_{PS}$ and $c_A$ be defined as in \eqref{csegnato} and \eqref{clambda}, respectively.  Then, there exists $\lambda^*>0$ such that for all $\lambda>\lambda^{*}$, it holds $0< c_A<c_{PS}$.
\end{lemma}

\begin{proof}
We take the open set $\Omega_H$ where $H$ is positive, see \eqref{H2:hyp}. Let $u_0\in D^{s,p}_A(\R^3,\C)\setminus \{0\}$ with support in
$\Omega_H$ such that $u_0\ge0$. As in the proof of Lemma \ref{mpt_geometry}, we have
\begin{equation}
    \cE(tu_0)= \frac{t^p}{p}[u_0]_{A,s,p}^p - \lambda\frac{t^q}{q}\int_{\R^3}H(x) |u_0|^q \dd x-\frac{t^{p^*_s}}{p^*_s}\int_{\R^3}K(x) |u_0|^{p^*_s} \dd x \to -\infty,
\end{equation}
as $t\to\infty$. Moreover $\cE(tu_0)\to 0^+$ as $t\to0^+$ .
Thus, there exists $t_\lambda>0$ such that
$$
    \max_{t\ge 0}\cE(tu_0)=\cE(t_\lambda u_0).
$$
In particular, we get
$$
    0=\frac{d}{dt}\Bigl[\cE(tu_0)\Bigr]_{t=t_\lambda}=t_\lambda^{p-1}[u_0]_{A,s,p}^p-\lambda t_\lambda^{q-1}\int_{\R^3}H(x) |u_0|^q\dd x - t_\lambda^{p^*_s-1}\int_{\R^3}K(x) | u_0|^{p^*_s} \dd x
$$
or, equivalently,
\begin{equation}\label{primozero}
    \lambda \int_{\R^3}H(x) |u_0|^q\dd x=
    \frac{1}{t_\lambda^{q-p}}[u_0]_{A,s,p}^p - t_\lambda^{p^*_s-q} \int_{\R^3}K(x) | u_0|^{p^*_s} \dd x
\end{equation}
for every $\lambda>0$. Since the support of $u_0$ is contained in $\Omega_H$, the left hand side of \eqref{primozero} is positive and it goes to $\infty$ if $\lambda\to\infty$. Thus, also the right hand side of \eqref{primozero} must go to $\infty$ if $\lambda\to\infty$. Hence, being $p<q<p^*_s$, necessarily $t_\lambda\to 0^+$ as $\lambda\to\infty$. From $\cE(t_\lambda u_0)\to 0^+$ as $t_\lambda\to0^+$ or equivalently when $\lambda\to\infty$, we can conclude that there exists $\lambda^{*}>0$ such that
$$
    \max_{t\ge0}\cE(tu_0)=\cE(t_\lambda u_0)<\frac{s}{3}\frac{S^{3/sp}}{ \|K\|_{L^\infty(\R^3, \CC)}^{3/(sp^*_s)}}=c_{PS},
$$
for every $\lambda>\lambda^{*}$. Therefore, $c_A<c_{PS}$
for all $\lambda>\lambda^{*}$, concluding the proof.
\end{proof}

\begin{remark}\label{q=pprob}
Let us consider the $p$-linear case, i.e., $q=p$. It is worth noting that, although the energy functional $\mathcal{E}$ exhibits a mountain pass geometry (see Lemma \ref{mpt_geometry}) and satisfies a compactness assumption (see Lemma \ref{lem5}), applying the techniques from \cite{bfk} to prove the final step, i.e. the claim of Lemma \ref{c<csegnato}, presents significant difficulties.

A primary distinction from \cite{bfk} lies in their functional setting, which involves the non-homogeneous Sobolev space. Furthermore, their nonlinearity consists of a critical term combined with a subcritical one satisfying the Ambrosetti-Rabinowitz condition; this allows them to establish the mountain pass geometry without any restriction on $\lambda$, which is a crucial point in this direction and differs from our Lemma \ref{mpt_geometry}. Consequently, their framework enables the study of the entire range of $sp^2$ relative to the dimension. In our setting, however, the use of the homogeneous Sobolev space and a $p$-linear term would necessitate restricting the analysis to the case $sp^2 < 3$, as in \cite{BonderSaintierSilva}.

Nevertheless, the presence of the magnetic term prevents us from handling even the $sp^2 < 3$ case. Specifically, the proof that the mountain pass level $c_M$ lies below the threshold $c_{PS}$ where a $p$-linear term is involved typically in a non-magnetic setting relies on the behavior of the minimizers for the best Sobolev constant (corresponding to \eqref{S} in our context) and requires $\lambda$ to be sufficiently large. In the magnetic setting, one must consider a magnetic perturbation of these minimizers belonging to $X$. This perturbation is incompatible with the techniques used in \cite{BonderSaintierSilva}. 

While one could theoretically conclude by taking $\lambda$ large (as in \cite{bfk}) if no upper bound on $\lambda$ were present, the magnetic term introduces a structural obstacle. In the literature, this is often addressed either by assuming an Ambrosetti-Rabinowitz type condition on the nonlinearity \cite{Ambrosio_schrodinger, bfk} or by imposing specific assumptions on the potential \cite{chabszulkin}.
\end{remark}

We are ready to prove Theorem \ref{main:theorem:q>p}. 

\begin{proof}[Proof of Theorem \ref{main:theorem:q>p}.]
Let $\cE:D^{s,p}_A(\R^3,\C) \to \R$ be the functional defined in \eqref{energy:functional:BN}. Assume  $\lambda>\lambda^*$, where $\lambda^*$ is defined in Lemma \ref{lem5},
By Lemma \ref{mpt_geometry} the functional satisfies the assumptions of Theorem \ref{mpthm}, hence there exists a $(PS)_{c_M}-$sequence, with $c_M>0$, $(u_n)_n$ which is bounded by Lemma \ref{lembound}. Hence, there exists $u \in D^{s,p}_A(\R^3,\C)$ such that $u_n \weakto u$ in $D^{s,p}_A(\R^3,\C)$.
Moreover, by Lemma \ref{lem5} the $(PS)_{c_M}-$condition holds since $c_M \le c_A< c_{PS}$ by Lemma \ref{c<csegnato} and Remark \ref{culambda}, yielding $u_n \to u$ in $D^{s,p}_A(\R^3,\C)$, which is a nontrivial critical point for $\cE$ with positive energy $c_M$.
\end{proof}

\subsection{Proof of Theorem \ref{main:theorem:q<p}}

In this section, we prove the existence result for $1<q<p$, i.e. Theorem \ref{main:theorem:q<p}, whose statement is given in the introduction, but first,
we recall briefly the definition of the genus inspired by \cite{GP87, struwe}. Let $Y$ be a real Banach space and let
$$
    \Sigma=\left\{ G\subset Y \backslash\left\{0\right\} | \,\ G \ \text{closed and symmetric} \quad u\in G\Rightarrow -u\in G\right\}.
$$
For any $G\in\Sigma$, the genus $\gamma\left( G\right)$ of $G$ is defined as the smallest integer $N$ so that there exists
$\Phi\in C\left( Y,\mathbb{R}^N\backslash\left\{0\right\}\right)$ such that $\Phi$ is odd and $\Phi\left(x\right)\ne 0$ for all
$x\in G$. We set $\gamma\left(\emptyset \right)=0$ and $\gamma\left(G\right)=\infty$ if there are no integers with the above property.

For the main properties of genus we refer to \cite[Proposition 1]{BBF}. In what follows we report the ones needed below, where $G,\bar G\in\Sigma$.
\begin{equation}\label{genussphere}
    \text{If } S^{N-1} \text{ is the unit sphere in } \R^N, \text{ then } \gamma(S^{N-1})=N.
\end{equation}
\begin{equation}\label{genussubset}
    \text{If $G\subset \bar G$ then $\gamma(G)\le\gamma(\bar G)$}
\end{equation}
\begin{equation}\label{genuscompact}
    \text{If $G$ is compact then $\gamma(G)<\infty$.}
\end{equation}
Now, we are ready to prove our main result in the $p$-sublinear case.

{\it Proof of Theorem \ref{main:theorem:q<p}.} Let $K_c=K_{c, \cE_\infty}=\left\{u\in \cD^{s,p}_A(\R^3,\CC):\, \cE_\infty(u)=c, \,\,\cE'_{\infty}(u)=0\right\}\,$
and take $m\in\mathbb N^+$. For $1\le j\le m$  define
$$
    c_{j}=\inf_{G\in\Sigma_{j}}\sup_{u\in G}\cE_\infty(u)
$$
where
$$
    \Sigma_{j}=\left\{G\subset \cD^{s,p}_A(\R^3,\CC)\backslash\left\{0\right\}: G \ \text{is closed in} \ \cD^{s,p}_A(\R^3,\CC),\,\, -G=G,\,\,\gamma(G)\ge j\right\}.
$$

We claim that
\begin{equation}\label{cj<0}-\infty<c_j<0\quad\text{for all } j\ge1.
\end{equation}
Since $\cE_\infty$ is bounded from below, hence $c_{j}>-\infty$ (that is why we consider $\cE_\infty$ instead of $\cE$). Thus, to reach the claim it is enough to prove that for all $j\in \mathbb{N}$, there is an
$\varepsilon_{j}=\varepsilon(j)>0$ such that
\begin{equation}\label{gammaeinfty}
    \gamma(\cE_\infty^{-\varepsilon_{j}})\ge j, \,\ \textnormal{where} \,\ \cE_\infty^{a}=\left\{u\in \cD^{s,p}_A(\R^3,\CC):\,
    \cE_\infty(u)\le a\right\}\quad \text{with} \quad a\in\mathbb R.
\end{equation}
Let $\Omega\subset\mathbb{R}^3$, $|\Omega|>0$,  be a bounded open set
in which  $H>0$, eventually $\Omega\subset \Omega_H$  where $\Omega_H$ is given in \eqref{H2:hyp}.
Extending functions $u$ in $\cD^{s,p}_{A,0}(\Omega,\C)$  by
$0$ outside $\Omega$, then $u\in \cD^{s,p}_{A}(\mathbb R^3,\C)$  and we can assume that $\cD^{s,p}_{A,0}(\Omega,\C)\subset \cD^{s,p}_A(\R^3,\CC)$.
Let $W_{j}$ be a $j$-dimensional subspace of $\cD^{s,p}_{A,0}(\Omega,\CC)$.
For every $v\in W_{j}$ with $[v]_{A,s,p}=1$, from the assumptions of
$H$ it is easy to see that there exists a $d_{j}>0$ such that
\begin{equation}\label{int1}
\int_{\Omega}H|v|^q \dd x\ge d_{j}.
\end{equation}
Take $t\in(0,T_{0})$, where $T_0$ appears in Lemma \ref{genus_geometry}, by Lemma \ref{genus_geometry}-$2)$ and by \eqref{K:hyp}, we arrive to
$$
    \cE_\infty(tv)=\cE(tv)=\frac{1}{p}t^{p}-\frac{\lambda t^{q}}{q}\int_{\Omega}H|v|^{q}\dd x-\frac{t^{p^{*}_s}}{p^{*}_s}\int_{\Omega}K|v|^{p^{*}_s}\dd x\le t^{q}\left(\frac{1}{p}
t^{p-q}-\frac{\lambda d_{j}}{q}\right),
$$
where in the last inequality we used \eqref{int1}.
Since $1<q<p$, we can choose $\varepsilon_j>0$ and $\tilde T_0<T_0$ such that $\cE_\infty(\tilde T_0v)<-\varepsilon_j$ for $v\in W_j$ with $[v]_{A,s,p}=1$.

Denote $S_{\tilde T_0}=\left\{v \in \cD^{s,p}_A(\R^3,\CC): [v]_{A,s,p}=\tilde T_0\right\},$ then $S_{\tilde T_0}\cap W_{j}\subset
\cE_\infty^{-\varepsilon_{j}}$. By \eqref{genussubset} and \eqref{genussphere} it follows
$\gamma(\cE_\infty^{-\varepsilon_{j}})\ge \gamma(S_{\tilde T_0}\cap W_{j})=j,$
which proves \eqref{gammaeinfty}. Consequently, $\cE_\infty^{-\varepsilon_{j}}\in\Sigma_{j}$, in turn
$c_j\le \sup_{u\in \cE_\infty^{-\varepsilon_{j}}}\cE_\infty(u)\le -\varepsilon_{j}<0.$
Thus the proof of  claim \eqref{cj<0} is concluded.

By \cite[Theorem 3.3]{GP87} then any $c_j$, with $j\in\mathbb N$, is a critical value for $\cE_\infty$ and if for some $j\in \mathbb{N}$ there is an $i\ge0$ such that
\begin{equation}\label{ci=}
c=c_{j}=c_{j+1}=\dots=c_{j+i},\quad\text{then}\quad \gamma(K_{c})\ge i+1.
\end{equation}
Assume $\lambda<\lambda_*:=\min\{\lambda_{*,1}, \lambda_{*,2}\}$, where $\lambda_{*,1}, \lambda_{*,2}$ are defined respectively in \eqref{lambdagenus} and \eqref{lambdacomp}. From Lemma \ref{compeinf} the truncated operator $\cE_\infty$ satisfies the $(PS)_{c_j}-$condition for all
$c_j(<0)$ and this implies that
$K_{c_j}$ is a compact set, hence $\gamma(K_{c_j})<\infty$ by virtue of \eqref{genuscompact}.
To complete the proof, we observe that for all $j\in \mathbb{N}^+$, we have
$\Sigma_{j+1}\subset\Sigma_{j}$ and $c_{j}\le c_{j+1}<0.$
Now two possible situations could occur.
\begin{itemize}
    \item If all $c_{j}$ are distinct, then $\gamma(K_{c_{j}})\ge 1$ from \eqref{ci=}, so that $K_{c_{j}}\neq\emptyset$ and
thus $(c_{j})_j$ is a sequence of distinct negative critical values of $\cE_\infty$ and of $\cE$ by Lemma \ref{compeinf}, thus a sequence of solutions to \eqref{main:equation} with negative energy is obtained, as required.
\item If for some $j_{0}$, there exists an $i\ge 1$ such that
$$
    c=c_{j_{0}}=c_{j_{0}+1}=\dots=c_{j_{0}+i},
$$
from \eqref{ci=} we have $\gamma(K_{c_{j_{0}}})\ge i+1>1$,
which shows that $K_{c_{j_{0}}}$ has infinitely many distinct elements, see Lemma 5.6 Chapter II in \cite{struwe}.
Also in this case, by Lemma \ref{compeinf}, we arrive to a sequence of solutions to \eqref{main:equation} with negative energy.

\end{itemize}

\qed

\appendix

\section{The magnetic concentration compactness}
\label{Appendix}
Despite most of the proof of Theorem \ref{Our:concentration:compactness} can be inherited from \cite[Theorem 1.1]{BonderSaintierSilva}, we decided, inspired also by \cite{ArioliSzulkin2003}, to report in this Appendix the full proof, suitable adapted to the magnetic case, to provide a complete result since, as far as we are aware, such a result for the fractional magnetic $p$-laplacian is not present in the literature.

\begin{proof}[Proof of Theorem \ref{Our:concentration:compactness}]
Since $u_n\rightharpoonup u$ in $D^{s,p}_A(\R^3,\C)$, then $(u_n)_n$ is bounded in $D^{s,p}_A(\R^3,\C)$. Thus, by \cite[Proposition 1.202]{FonsecaLeoni}, then, up to subsequences, there exist $\nu,\mu$ bounded positive measures such that 
$$
    |D^s_A u_n|^p\dd x\rightharpoonup \mu,\qquad |u_n|^{p^*_s}\dd x\rightharpoonup \nu,
$$
where $D^s_A$ is defined in \eqref{definition:DsA:seminorm}.

First, we consider the case $u=0$. In this case, we divide the proof into intermediate steps.

{\bf Step $\mathbf{1}$.} Prove that
\begin{equation}\label{reverse:holder}
    S_A^\frac{1}{p}\left(\int_{\R^3} |\phi|^{p^*_s}\dd\nu\right)^\frac{1}{p^*_s} \le \left(\int_{\R^3} |\phi|^p\dd\mu\right)^\frac{1}{p}\qquad \text{for any}\quad  \phi\in C^\infty_c(\R^3,\R).
\end{equation}

To prove \eqref{reverse:holder}, observe that, given $\phi\in C^\infty_c(\R^3,\R)$, applying the Sobolev inequality
\eqref{magnetic:Sobovel:inequality:Dspa} we get
$$
    S_A^\frac{1}{p} \|\phi u_n\|_{L^{p^*_s}(\R^3,\C)}\le \| D^s_A(\phi u_n)\|_{L^p(\R^3,\C)}.
$$
By the definition of $\nu$, it follows that $\|\phi u_n\|_{L^{p^*_s}(\R^3,\C)}\to \left(\int_{\R^3} |\phi|^{p^*_s}\dd\nu\right)^\frac{1}{p^*_s}$ as  $n\to\infty$.

For the right-hand side, by Minkowski inequality, we observe that
\begin{align*}
    \|D^s_A(\phi u_n)\|_{L^p(\R^3,\C)} &\le \left(\int_{\R^3\times\R^3}|\phi(x)|^p\frac{|\expAminus u_n(x)-u_n(y)|^p}{|x-y|^{3+sp}}\dd x\dd y\right)^{\frac1p}\\
    &\quad+\left(\int_{\R^3\times\R^3}|u_n(y)|^p\frac{|\phi(x)-\phi(y)|^p}{|x-y|^{3+sp}}\dd x\dd y\right)^{\frac1p}.
\end{align*}

Hence, we get
$$
    \|D^s_A(\phi u_n)\|_{L^p(\R^3,\C)} \leq \left(\int_{\R^3} |\phi(x)|^p|D^s_A u_n(x)|^p\dd x\right)^{\frac1p}+\left(\int_{\R^3} |u_n(x)|^p |D^s \phi(x)|^p\dd x\right)^{\frac1p}.
$$

Now, from \cite[Lemma 2.2]{BonderSaintierSilva}, the weight $w(x) := |D^s\phi(x)|^p$ satisfies the hypotheses of Lemma \ref{BoSaSi:cpt:lemma}, and hence \eqref{strongconvpesata} holds.
Therefore, also by the definition of $\mu$, we get
$$
    \limsup_{n\to\infty} \|D^s_A(\phi u_n)\|_{L^p(\R^3,\C)}\le \left(\int_{\R^3} |\phi|^p\dd\mu\right)^{\frac1p}.
$$
This concludes the proof of the reverse H\"older inequality \eqref{reverse:holder}.

{\bf Step $\mathbf{2}$.} Existence of the atomic parts of the measures $\mu$ and $\nu$ and their relation.

From \eqref{reverse:holder} it follows exactly as in \cite[Lemma 1.2]{Lions1985} that there exists a countable set $I$, points $(x_i)_{i\in I}\subset \R^3$ and positive weights $(\nu_i)_{i\in I}, (\mu_i)_{i\in I} \subset \R$ such that
$$
    \nu = \sum_{i\in I} \nu_i \delta_{x_i},\qquad \mu\ge \sum_{i\in I} \mu_i \delta_{x_i}.
$$

Now, to prove the relation \eqref{relation:nu:mu:atomic},
we take $\varphi\in C^\infty_c(\R^3,\R)$ be such that $0\le \varphi\le 1$, $\varphi(0)=1$, $\supp\varphi = B_1(0)$ and for any given $\ve>0$ we consider the rescaled functions $\varphi_{i,\ve}(x) = \varphi(\frac{x-x_i}{\ve})$.

Without loss of generality we may assume that $x_i=0$ and write $\varphi_\eps=\varphi_{i,\ve}$. From \cite[Corollary 2.3]{BonderSaintierSilva} it follows that
\begin{equation}\label{exact:decay}
|D^s \varphi_\ve(x)|^p \le C \min\{\ve^{-sp}; \ve^3 |x|^{-(3+sp)}\},
\end{equation}
which implies that $|D^s\varphi_\ve|^p$ satisfies the hypotheses of Lemma \ref{BoSaSi:cpt:lemma}, therefore, arguing as in the proof of the reverse H\"older inequality \eqref{reverse:holder}, one arrives at
$$
    S_A^\frac1p\left(\int_{\R^3} |\varphi_\ve|^{p^*_s}\dd\nu\right)^\frac{1}{p^*_s} \le \left(\int_{\R^3} |\varphi_\ve|^p\dd\mu\right)^\frac1p + \left(\int_{\R^3} |u|^p |D^s\varphi_\ve|^p\dd x\right)^\frac1p.
$$
Now, $\int_{\R^3} |\varphi_\ve|^{p^*_s}\dd\nu\ge \nu_i=\int_{\R^3}\delta_{x_i} \dd \nu$ and $\int_{\R^3} |\varphi_\ve|^p\dd\mu\le\int_{B_\ve(0)} \dd\mu \to \mu_i$  as $\ve\to 0$. Arguing exactly as in the proof of \cite[Theorem 1.1]{BonderSaintierSilva}, since $u\in L^{p^*_s}(\R^3,\CC)$, the following holds
\begin{equation}
\int_{\R^3} |u|^p |D^s\varphi_\ve|^p\dd x\to 0\quad \text{as } \ve\to 0.
\end{equation}
This concludes the proof of \eqref{relation:nu:mu:atomic}.

{\bf Step $\mathbf{3}$.} Concentration at infinity.

It remains to see \eqref{un:crit:inf}--\eqref{relation:nu:mu:infty}. Consider a smooth function $\eta\colon [0,+\infty)\to [0,1]$ such that $\eta\equiv 0$ in $[0,1]$ and $\eta\equiv 1$ in $[2,+\infty)$. Then $\eta_R(x):=\eta(|x|/R)$ is smooth and satisfies $\eta_R(x)=1$ for $|x|\ge 2R$, $\eta_R(x)=0$ for $|x|\le R$ and $0\leq\eta_R(x)\leq1$. We split
\begin{equation}\label{Ds:uk:split}
    \int_{\R^3}|D^s_A u_n|^{p}\dd x = \int_{\R^3}|D^s_A u_n|^{p}\eta_R^p\dd x + \int_{\R^3}|D^s_A u_n|^{p}(1-\eta_R^p)\dd x
\end{equation}
and we observe that
$$ 
    \int_{|x|>2R}|D^s_A u_n|^{p}\dd x \leq \int_{\R^3} |D^s_A u_n|^{p}\eta_R^p\dd x \leq \int_{|x|>R} |D^s_A u_n|^p \dd x
$$
so that, arguing similarly for $|u_n|^{p^*_s}$, there exist
\begin{equation}\label{mu:nu:inf}
    \mu_\infty = \lim_{R\to\infty}\limsup_{n \to +\infty} \int_{\R^3} |D^s_A u_n|^p \eta_R^p\dd x \quad \text{ and } \quad 
    \nu_\infty = \lim_{R\to\infty}\limsup_{n \to +\infty} \int_{\R^3} |u_n|^{p^*_s} \eta_R^{p^*_s}\dd x.
\end{equation}
On the other hand, since $1-\eta_R^p$ is smooth with compact support, and by the definition of $\mu$, for $R$ fixed, we get
$$ 
    \lim_{n\to \infty}  \int_{\R^3} (1-\eta_R^p) |D^s_A u_n|^p\dd x = \int_{\R^3} (1-\eta_R^p)\dd\mu.  
$$

Since $\eta_R\to 0$ pointwise and $\mu$ is a bounded nonnegative measure, it follows from the Dominated Convergence Theorem that $\lim_{R\to \infty}\int_{\R^3} \eta_R^p\dd\mu = 0$. Hence, we obtain
\begin{equation}\label{mu:Rn:measure}
\lim_{R\to\infty}\limsup_{n \to +\infty} \int_{\R^3} (1-\eta_R^p) |D^s_A u_n|^p\dd x = \int_{\R^3} \dd \mu.
\end{equation}
Plugging \eqref{mu:nu:inf} and \eqref{mu:Rn:measure} into \eqref{Ds:uk:split} yields
\eqref{Dspa:inf}. The proof of \eqref{un:crit:inf} is similar.

By applying \eqref{magnetic:Sobovel:inequality:Dspa}, we have
$$
    S_A^\frac{1}{p}\|u_n \eta_R\|_{L^{p^*_s}(\R^3,\C)} \leq \| D^s_A(u_n \eta_R)\|_{L^p(\R^3,\C)}
$$ 
and as before, we have
$$
    \left(\int_{\R^3} |D^s_A(u_n \eta_R)|^p\dd x\right)^{1/p}\leq \left(\int_{\R^3} |\eta_R(x)|^p|D^s_A u_n(x)|^p\dd x\right)^{\frac{1}{p}}+\left(\int_{\R^3} |u_n(x)|^p |D^s \eta_R(x)|^{p}\dd x\right)^{\frac{1}{p}}.
$$

Arguing again as in \cite{BonderSaintierSilva}, we obtain
\begin{equation}\label{uk:Ds:to:0}
\lim_{R\to\infty} \lim_{n \to +\infty} \int_{\R^3}|u_n(x)|^p|D^s\eta_R(x)|^p\dd x = 0,
\end{equation}
which gives \eqref{relation:nu:mu:infty}

{\bf Step $\mathbf{4}$.} The case $u\neq 0$.

Let $v_n=u_n-u$, then $v_n \weakto 0$ in $D^{s,p}_A(\R^3,\C)$. Therefore, repeating the proof for $v_n$, together with the Brezis-Lieb Lemma \cite[Theorem 1]{BrezisLieb}, leads to the conclusion.

\end{proof}

\section*{Acknowledgments}

The authors express their gratitude to the {\em Centre International de Rencontres Math\'ematiques (CIRM)}, where an important part of this work was carried out during the Research program titled \textit{When Locality and Compactness Vanish: A Magnetic Perspective} from 3 - 14 November, 2025.
The authors are members of the {\em Gruppo Nazionale per l'Analisi Ma\-te\-ma\-ti\-ca, la Probabilit\`a e le loro Applicazioni} (GNAMPA) of the {\em Istituto Nazionale di Alta Matematica} (INdAM) and are partially supported by INdAM-GNAMPA Project 2026 titled \textit{Structural degeneracy and criticality in (sub)elliptic PDEs} (E53C25002010001). L.B. is partially supported by the Deutsche Forschungsgemeinschaft (DFG, German Research Foundation) - Project-ID 258734477 - SFB 1173.

\bibliography{Bibliography}
\bibliographystyle{abbrv}

\end{document}